\documentclass[11pt,makeidx]{amsart}
\usepackage{amsfonts,amssymb,amsmath,amscd, bbm, amsrefs}
\usepackage{tikz}
\usetikzlibrary{positioning}
\font\cyr=wncyr10

\newcommand{\A}{{\mathbb A}}
\newcommand{\C}{{\mathbb C}}
\newcommand{\F}{{\mathbb F}}
\newcommand{\ff}{{\mathbbm f}}

\newcommand{\Q}{{\mathbb Q}}
\newcommand{\Ql}{{\Q_\ell}}
\newcommand{\Qp}{{\Q_p}}
\newcommand{\bQ}{{\overline\Q}}
\newcommand{\bQl}{\bQ_\ell}
\newcommand{\bQp}{\bQ_p}

\newcommand{\R}{{\mathbb R}}
\newcommand{\T}{{\mathbb T}}

\newcommand{\Z}{{\mathbb Z}}
\newcommand{\Zp}{{\Z_p}}

\newcommand{\gra}{{\mathfrak a}}

\newcommand{\grb}{{\mathfrak b}}

\newcommand{\grL}{{\mathfrak L}}
\newcommand{\grm}{{\mathfrak m}}

\newcommand{\grp}{{\mathfrak p}}

\newcommand{\eps}{{\epsilon}}

\newcommand{\kap}{{\kappa}}
\newcommand{\vphi}{{\varphi}}

\newcommand{\CC}{{\mathcal C}}

\newcommand{\CL}{{\mathcal L}}
\newcommand{\CM}{{\mathcal M}}
\newcommand{\CN}{{\mathcal N}}
\renewcommand{\O}{{\mathcal O}}
\newcommand{\CP}{{\mathcal P}}
\newcommand{\CS}{{\mathcal S}}

\newcommand{\CX}{{\mathcal X}}

\newcommand{\alg}{\mathrm{alg}}
\newcommand{\Aut}{\mathrm{Aut}}

\newcommand{\cyc}{\mathrm{cyc}}
\renewcommand{\div}{\mathrm{div}}

\newcommand{\frob}{\mathrm{frob}}
\newcommand{\Gal}{\mathrm{Gal}}
\newcommand{\GL}{\mathrm{GL}}
\newcommand{\Hom}{\mathrm{Hom}}
\newcommand{\im}{\mathrm{im}}
\newcommand{\isoarrow}{\stackrel{\sim}{\rightarrow}}
\newcommand{\isom}{\cong}

\newcommand{\opt}{\mathrm{opt}}
\newcommand{\ord}{\mathrm{ord}}

\newcommand{\sat}{\mathrm{sat}}
\newcommand{\Sel}{\mathrm{Sel}}
\newcommand{\Sha}{\hbox{\cyr X}}
\newcommand{\SL}{\mathrm{SL}}
\newcommand{\trace}{\mathrm{trace}\,}
\newcommand{\ur}{\mathrm{ur}}

\makeatletter
\@addtoreset{equation}{subsection}

\makeatother

\newtheorem{lem}[subsubsection]{\bf Lemma}
\newtheorem{thm}[subsubsection]{\bf Theorem}
\newtheorem{prop}[subsubsection]{\bf Proposition}
\newtheorem{conj}[subsubsection]{\bf Conjecture}

\newtheorem{innercustomthm}{Theorem}
\newenvironment{customthm}[1]
  {\renewcommand\theinnercustomthm{#1}\innercustomthm}
  {\endinnercustomthm}

\theoremstyle{definition}
\newtheorem{rmk}[subsubsection]{\it Remark}

\author{Christopher Skinner}

\address{
 Department of Mathematics\\
 Princeton University\\
  Fine Hall, Washington Road \\
 Princeton, NJ 08544-1000\\
 USA}

\title[Multiplicative reduction and the cyclotomic main conjecture 
for $\GL_2$]
{Multiplicative reduction and the cyclotomic \\ main conjecture 
for $\GL_2$
}

\begin{document}
\setcounter{section}{0}
\begin{abstract}
We show that the cyclotomic Iwasawa--Greenberg Main Conjecture holds for 
a large class of modular forms with multiplicative reduction at $p$, extending
previous results for the good ordinary case. In fact, the multiplicative case
is deduced from the good case through the use of Hida families and a simple Fitting ideal argument.
\end{abstract}
\maketitle

\section{Introduction}

The cyclotomic Iwasawa--Greenberg Main Conjecture was established in \cite{SU-MCGL}, in 
combination with work of Kato \cite{Kato}, for a large
class of newforms $f\in S_k(\Gamma_0(N))$ that are ordinary at an odd prime $p\nmid N$, subject to 
$k\equiv 2 \pmod {p-1}$ and certain conditions on the mod~$p$ Galois representation associated
with $f$. The purpose of this note is to extend this result to the case where
$p\mid N$ (in which case $k$ is necessarily equal to $2$). 

Recall that the coefficients $a_n$ of the $q$-expansion $f=\sum_{n=1}^\infty a_nq^n$ of $f$ at the cusp at 
infinity (equivalently, the Hecke eigenvalues of $f$) are algebraic integers that generate a finite extension 
$\Q(f)\subset\C$ of $\Q$. Let $p$ be an odd prime and let $L$ be a finite extension of
the completion of $\Q(f)$ at a chosen prime above $p$ (equivalently, let $L$ be a finite
extension of $\Qp$ in a fixed algebraic closure $\bQp$ of $\Qp$ that contains the image of a chosen 
embedding $\Q(f)\hookrightarrow\bQp$). Suppose that $f$ is ordinary at $p$ with respect to $L$ 
in the sense that $a_p$ is a unit in the ring of integers $\O$ of $L$. Then the $p$-adic $L$-function 
$\CL_f$ of $f$ is an element of the Iwasawa algebra $\Lambda_\O=\O[\![\Gamma]\!]$,
where $\Gamma=\Gal(\Q_\infty/\Q)$ is the Galois group of the cyclotomic $\Zp$-extension
$\Q_\infty$ of $\Q$.  A defining property of $\CL_f$ is that it interpolates normalized special values of the 
$L$-function of $f$ twisted by Dirichlet characters associated with finite-order characters of 
$\Gamma$. The Iwasawa--Greenberg Selmer group $\Sel_{\Q_\infty,L}(f)$, defined
with respect to the $p$-adic Galois representation $V_f$ of $f$ over $L$ - a two-dimensional $L$-vector space
- and a Galois-stable $\O$-lattice $T_f\subset V_f$, is a discrete, cofinite $\Lambda_\O$-module, and
the Iwasawa--Greenberg characteristic ideal $Ch_L(f)\subset\Lambda_\O$ is the characteristic $\Lambda_\O$-ideal
of the Pontryagin dual $X_{\Q_\infty,L}(f)$ of $\Sel_{\Q_\infty,L}(f)$. 
The Iwasawa--Greenberg Main Conjecture for $f$ then asserts that there is an equality of ideals 
$Ch_L(f) = (\CL_f)$ in $\Lambda_\O\otimes_\Zp\Qp$ and even in $\Lambda_\O$ if $T_f$ is residually irreducible.

\begin{customthm}{A}\label{thmA} Let $p\geq 3$ be a prime.
Let $f\in S_k(\Gamma_0(N))$ be a newform and let $L$ and $\O$ be as above
and suppose $f$ is ordinary at $p$ with respect to $L$.
If
\begin{itemize}
\item[\rm (i)] $k\equiv 2\pmod{p-1}$;
\item[\rm (ii)] the reduction $\bar\rho_f$ of 
the representation $\rho_f:\Gal(\bQ/\Q)\rightarrow \Aut_\O(T_f)$ modulo the maximal ideal of $\O$ is irreducible;
\item[\rm (iii)] there exists a prime $q\neq p$ such that $q\mid \mid N$ and
$\bar\rho_f$ is ramified at $q$,
\end{itemize}
then $Ch_{L}(f)=(\CL_f)$ in $\Lambda_{\O}$. That is,
the Iwasawa--Greenberg Main Conjecture is true.
\end{customthm}

When $p\nmid N$ this is just Theorem 1 of \cite{SU-MCGL}\footnote{In order to conclude that the equality
holds in $\Lambda_\O$ and not just $\Lambda_\O\otimes_\Zp\Qp$, Theorem 1 in \cite{SU-MCGL} requires 
that $\rho_f$ have an $\O$-basis with respect to which the image contains $\SL_2(\Zp)$. 
But as we explain in Section~\ref{MCsection}, hypotheses (ii) and (iii) of Theorem \ref{thmA} are enough for the arguments.
We also explain that the reference to \cite{Vatsal} in \cite{SU-MCGL} should have been augmented with a reference to 
\cite{Chida-Hsieh}.}. 
When $p\mid N$, in which case the ordinary
hypothesis forces $p\mid \mid N$ and $k=2$, this is not an immediate consequence of the results in \cite{SU-MCGL}, 
as this case is excluded from Kato's divisibility theorem \cite[Thm.~17.4]{Kato}, which is a crucial 
ingredient in the deduction of the main conjecture from the main results in \cite{SU-MCGL}.  However,
as we explain in this note, the main conjecture in the case $p\mid N$ can be deduced from knowing it
when $p\nmid N$.

Having the cyclotomic main conjecture in hand, one obtains results toward special value formulas.
For example:

\begin{customthm}{B}\label{thmB}  Let $p\geq 3$ be a prime.
Let $f\in S_2(\Gamma_0(N))$ be a newform and let $L$ and $\O$ be as above
and suppose $f$ is ordinary.
Suppose also that 
\begin{itemize}
\item[(i)] the reduction $\bar\rho_f$ of 
the representation $\rho_f:\Gal(\bQ/\Q)\rightarrow \Aut_\O(T_f)$ modulo the maximal ideal of $\O$ is irreducible;
\item[(ii)] there exists a prime $q\neq p$ such that $q\mid \mid N$ and
$\bar\rho_f$ is ramified at $q$;
\item[(iii)] if $p\mid N$ and $a_p = 1$, then the $\grL$-invariant $\grL(V_f)\in L$ is nonzero.
\end{itemize}
Let 
$$
L^\alg(f,1) = \frac{L(f,1)}{-2\pi i\Omega_f^+}.
$$
Then 
$$
\#\O/(L^\alg(f,1)) = \#\Sel_{L}(f) \cdot \prod_{\ell} c_\ell(T_f).
$$
In particular, if $L(f,1)=0$, then $\Sel_L(f)$ has $\O$-corank at least one.
\end{customthm}

\noindent Here $\Omega_f^+$ is one of two canonical periods associated with $f$
as in \cite[\S3.3.3]{SU-MCGL} (and well-defined up to an element of
$\O^\times\cap \Q(f)$), 
$\Sel_L(f)$ is the Selmer group associated by Bloch-Kato to the 
Galois lattice $T_f$, $c_\ell(T_f)$ 
is the Tamagawa factor at $\ell$ of $T_f$ (and equals $1$ unless $\ell\mid N$), and $\grL(V_f)$ is the $\CL$-invariant
of a modular form $f$ (or of $V_f$) with split multiplicative reduction at $p$ introduced by
Mazur, Tate, and Teitelbaum in \cite{MTT} (see also \cite[\S3]{GrSt}). It is conjectured that  
$\grL(V_f)$ is always nonzero; this is known if $f$ is the modular form associated to an elliptic curve, but in general
it is an open question.

As a special case of Theorem \ref{thmB}, obtained by taking $f$ to be the newform associated with an elliptic curve 
$E$ over $\Q$, we have:

\begin{customthm}{C}\label{thmC}
Let $E$ be an elliptic curve over $\Q$ with good ordinary or multiplicative reduction at a prime $p\geq 3$. Suppose that
\begin{itemize}
\item[(i)] $E[p]$ is an irreducible $\Gal(\bQ/\Q)$-representation;
\item[(ii)] there exists a prime $q\neq p$ at which $E$ has multiplicative reduction and $E[p]$ is ramified.
\end{itemize}
If $L(E,1)\neq 0$ then
$$
\left|\frac{L(E,1)}{\Omega_E}\right|_p^{-1} = \left| \#\Sha(E)\prod_{\ell} c_\ell(E)\right|_p^{-1},
$$
and if $L(E,1)=0$ then $\Sel_{p^\infty}(E)$ has $\Zp$-corank at least one.
\end{customthm}
\noindent Here, $\Omega_E$ is the N\'eron period of $E$, 
$\Sha(E)$ is the Tate-Shafarevich group of $E/\Q$, and the $c_\ell(E)$ are the Tamagawa numbers of $E$. 
In particular, $c_\ell(E)$ is the order of the group of irreducible components of 
the special fiber of the N\'eron model of $E$ over $\Z_\ell$.

Our proof of Theorem \ref{thmA} is relatively simple. Let $N=pM$. We first make two reductions: (1) it suffices to prove
the theorem with the field $L$ replaced by any finite extension, and (2) it suffices to prove the 
equality $Ch_L^\Sigma(f) = (\CL^\Sigma_f)$, where $\Sigma$ is any finite set of primes 
containing all $\ell\mid N$, $\CL_f^\Sigma$ being the incomplete $p$-adic 
$L$-function with the Euler factors at primes in $\Sigma$ different from $p$ removed and $Ch^\Sigma_L(f)$ being
the characteristic ideal of the Pontryagin dual $X^\Sigma_{\Q_\infty,L}(f)$ of the Iwasawa--Greenberg Selmer group
$\Sel^\Sigma_{\Q_\infty,L}(f)$ with all conditions at primes
in $\Sigma$ different from $p$ relaxed.
Then we exploit Hida theory to deduce that one can choose $L$ so that for each integer $m>0$
there exists a newform $f_m\in S_{k_m}(\Gamma_0(M))$ with $k_m\equiv k \pmod{(p-1)p^m}$,
$\Q(f_m)\subset L$ and $f_m$ ordinary at $p$ with respect to $L$, and the ordinary $p$-stabilization $f^*_m$ of 
$f_m$ satisfies
$f_m^* \equiv f \pmod{p^m}$ in the sense that the $q$-expansions (which have coefficients in $\O$)
are congruent modulo $p^m$.  Furthermore, as a consequence of the existence of the `two-variable' 
$p$-adic $L$-function associated to a Hida family we also have  
$\CL_{f_m}^\Sigma\equiv \CL_f^\Sigma \pmod{p^m\Lambda_\O}$.  An argument of Greenberg shows
that since $X^\Sigma_{\Q_\infty,L}(f_m)$ is a torsion $\Lambda_\O$-module, it has no non-zero 
finite-order $\Lambda_\O$-submodules. From this it follows that $Ch^\Sigma_{L}(f_m)$ equals
the $\Lambda_\O$-Fitting ideal $F_L^\Sigma(f_m)$ of $X^\Sigma_{L}(f_m)$.
The congruence $f_m^*\equiv f \pmod{p^m}$ implies that $\Sel^\Sigma_{\Q_\infty,L}(f)[p^m]
\cong \Sel^\Sigma_{\Q_\infty,L}(f_m)[p^m]$, so comparing Fitting ideals yields 
$$
(F_L^\Sigma(f), p^m) = (F_L^\Sigma(f_m), p^m) = (Ch_L^\Sigma(f_m), p^m)\subset \Lambda_\O.
$$
From the main conjecture for $f_m$ (the congruence $f_m^*\equiv f \pmod p$ ensures that 
the hypotheses of Theorem \ref{thmA} also hold for $f_m$) and the congruence modulo 
$p^m$ of $p$-adic $L$-functions we then have 
$$
(F_L^\Sigma(f), p^m) = (Ch_L^\Sigma(f_m), p^m) = (\CL_{f_m}^\Sigma, p^m) = 
(\CL_{f}^\Sigma, p^m)\subset \Lambda_\O
$$
for all integers $m>0$. This, together with the non-vanishing of the $p$-adic $L$-function $\CL_f^\Sigma$, implies
that $F_L^\Sigma(f)\neq 0$ and hence that $X^\Sigma_{\Q_\infty,L}(f)$ is a torsion $\Lambda_\O$-module.
The earlier argument of Greenberg then gives $Ch^\Sigma_L(f)  = F^\Sigma_L(f)$, and so
$(Ch^\Sigma_L(f),p^m) = (\CL^\Sigma(f),p^m) \subset \Lambda_\O$ for all $m>0$.
As $Ch_L^\Sigma(f)\subset\Lambda_\O$ is a principal ideal, it then easily follows that $Ch^\Sigma_L(f) = (\CL_f^\Sigma)$, 
proving Theorem \ref{thmA}.

The deduction of Theorem \ref{thmB} from Theorem \ref{thmA} follows an argument of 
Greenberg from~\cite{Gr-CIME}.

In addition to extending the main conjecture to the case of multiplicative reduction, our motivation for writing this note
was in part to provide an explicit reference for the expression for the special
value $L^\alg(f,1)$ in terms of the size of Selmer groups that is required for the arguments in \cite{WZ-GZform} 
and, by including the multiplicative reduction case, also provide an important ingredient for the extension of the main results of \cite{WZ-GZform}
to cases of multiplicative reduction. Additional motivation for the latter stems from the author's collaboration
with Manjul Bhargava and Wei Zhang to provide lower bounds on the proportion of elliptic curves that satisfy the rank part of the Birch--Swinnerton-Dyer conjecture.

While preparing this note the author learned of Olivier Fouquet's work \cite{OF-Tam} on the equivariant Tamagawa number conjecture
for motives of modular forms. That work should provide another means for deducing Theorem \ref{thmA} in the case $p\mid N$ from the main results\footnote{But see also footnote 1, especially as the main results in \cite{OF-Tam} rely on Theorem \ref{thmA} as stated, at least for the $p\nmid N$ case.} in \cite{SU-MCGL} as well as some additional weakening of the conditions on primes away from $p$. The deduction of Theorem \ref{thmA} for $p\mid N$ in this paper uses no more machinery than already developed in \cite{SU-MCGL} or than is required for our 
deduction of Theorem \ref{thmB}.

\noindent{\it Acknowledgments.} 
The author thanks Xin Wan for helpful conversations and Manjul Bhargava for helpful comments. 
The author's work was partially supported by the grants  DMS-0758379 and 
DMS-1301842 from the National Science Foundation.

\section{Gathering the pieces}\label{gather}

In this section we recall the various objects that go into the Iwasawa--Greenberg Main Conjecture for modular forms,
some of their properties, and some useful relations. Throughout $p$ is a fixed odd prime.

Let $\bQ\subset\C$ be the algebraic closure of $\Q$ and let $G_\Q=\Gal(\bQ/\Q)$. 
For each prime $\ell$, let $\bQl$ be a fixed algebraic closure of $\bQl$. 
For each $\ell$ we also fix an embedding $\bQ\hookrightarrow\bQl$, which identifies
$G_\Ql=\Gal(\bQl/\bQ)$ with a decomposition subgroup in $G_\Q$; let $I_\ell\subset G_\Ql$
be the inertia subgroup. Let $\frob_\ell \in G_\Ql$ be (a lift of) an arithmetic Frobenius
element.

Let $\eps:G_\Q\rightarrow \Z_p^\times$ be the $p$-adic cyclotomic character. This is
just the projection to $\Gal(\Q[\mu_{p^\infty}]/\Q)$, the latter being canonically isomorphic
to $\Z_p^\times$. Similarly, let $\omega:G_\Q\rightarrow \Z_p^\times$
be the mod $p$ Teichm\"uller character. This is just  the composition of the reduction of
$\eps$ mod $p$ and the multiplicative homomorphism 
$(\Z/p\Z)^\times\hookrightarrow \Z_p^\times$ defined by the Teichm\"uller lifts.

Let $\Q_\infty\subset\Q[\mu^{p\infty}]\subset \bQ$ be the cyclotomic $\Zp$-extension of $\Q$. That is, $\Q_\infty$ is
the unique abelian extension of $\Q$ such that $\Gamma=\Gal(\Q_\infty/\Q)\cong \Zp$. 
Let $\gamma\in \Gamma$ be a fixed topological generator. As 
$\Gal(\Q[\mu_{p^\infty}]/\Q) \isoarrow \Gal(\Q[\mu_p]/\Q)\times\Gamma$, there is a
lift $\tilde\gamma$ of $\gamma$ to $\Gal(\Q[\mu_{p^\infty}]/\Q)$ identified with $(1,\gamma)$,
and we let $u=\eps(\tilde\gamma)\in \Z_p^\times$.

\subsection{Galois representations and (ordinary) newforms}

Let $f\in S_k(\Gamma_0(N))$ be a newform. Let $\Q(f)\subset\C$ be
the finite extension of $\Q$ generated by the Fourier coefficients $a_n(f)$
of the $q$-expansion $f=\sum_{n=0}^\infty a_n(f)q^n$ of $f$ at the cusp at infinity
(equivalently, the field obtained by adjoining the eigenvalues of the action of the
usual Hecke operators on $f$). Fix an embedding $\Q(f)\hookrightarrow \bQp$ and let
$L\subset \bQp$ be a finite extension of $\Qp$ containing the image of $\Q(f)$. 
Let $\O$ be the ring of integers of $L$ (the valuation ring), let $\grm$ be its maximal
ideal, and let $\kap=\O/\grm$ be its residue field.

Associated with $f$ and $L$ (and the embedding $\Q(f)\hookrightarrow L$) is a 
two-dimensional $L$-space $V_f$ and an absolutely irreducible continuous $G_\Q$-representation 
$\rho_f:G_\Q\rightarrow \Aut_L(V_f)$ such that $\rho_f$ is unramified at all 
primes $\ell\nmid Np$ and $\det(1-X\cdot\rho_f(\frob_\ell)) = 1 - a_\ell(f) X + \ell^{k-1}X^2$
for such $\ell$. 
In particular, $\trace\rho_f(\frob_\ell) = a_\ell(f)$ if $\ell\nmid pN$, and $\det\rho_f = \eps^{k-1}$.

Let $T,T'\subset V_f$ be two $G_\Q$-stable $\O$-lattices. Let $\bar\rho$ and $\bar\rho'$
denote, respectively, the two-dimensional $\kap$-representations $T/\grm T$ and $T'/\grm T'$.
The following lemma is well-known, but we include it for later reference.
\begin{lem}\label{lattice-lem}\hfill
\begin{itemize} 
\item[\rm (a)] If $\bar\rho$ or $\bar\rho'$ is irreducible, then $\bar\rho$ and $\bar\rho'$ are equivalent as $\kap$-representations.
In particular, $\bar\rho$ is irreducible if and only if $\bar\rho'$ is irreducible.
\item[\rm (b)] If $\bar\rho$ or $\bar\rho'$ is irreducible, then there exists $a\in L^\times$ such that
$T = aT'$.
\end{itemize}
\end{lem}

\begin{proof} Replacing $T'$ with some $\O$-multiple, we may 
assume that $T'$ is a sublattice of $T$. Then $T/T'\cong \O/\grm^n\times \O/\grm^m$ with $n\leq m$.
Let $\varpi$ be a uniformizer of $\O$ (a generator of $\grm$). Then 
$\varpi^nT/(T'+\varpi^{n+1}T)\cong \O/\grm^{\min(1,m-n)}$
is a $G_\Q$-stable quotient of $T/\grm T \cong \varpi^nT/\varpi^{n+1}T$ of at most one-dimension over $k$.
If $\bar\rho$ is irreducible, then this quotient must be trivial and so $m-n=0$ and $T'= \varpi^n T$, in which
case $T'/\grm T' \cong \varpi^n T/\varpi^{n+1}T \cong T/\grm T$ as $G_\Q$-representations over $\kap$. Reversing 
the roles of $T$ and $T'$ in this argument then yields the lemma.
\end{proof}

We then define $\bar\rho_f$ to be the $\kap$-representation $T/\grm T$ of $G_\Q$ for a 
Galois-stable $\O$-lattice
$T\subset V_f$. By the above lemma, if $\bar\rho_f$ is irreducible for some choice of $T$, then it is irreducible for
any choice of $T$, and the equivalence class of $\bar\rho_f$ is independent of $T$. Of course, it is not difficult
to show that the semisimplification of $\bar\rho_f$ is independent of $T$ even when $\bar\rho_f$ is not irreducible,
but will not need this.

Suppose $k\geq 2$ and $f$ is ordinary with respect to the embedding $\Q(f)\hookrightarrow L$. That is, $a_p(f)\in \O^\times$.
Then $V_f$ has a unique $G_\Qp$-stable $L$-line $V_f^+\subset V_f$ such that $G_\Qp$ acts on $V_f^+$ via the character
$\alpha_f^{-1}\eps^{k-1}$, where $\alpha_f:G_\Qp\rightarrow \O^\times$ is the unique unramified character
such that $\alpha_f(\frob_p)$ equals the (unit) root $\alpha_p$ in $\O^\times$ of the polynomial $x^2-a_p(f)x + p^{k-1}$ if $p\nmid N$ and $\alpha_f(\frob_p) = a_p(f)$ if $p\mid  N$.
(Note that the reduction of the polynomial $x^2-a_p(f)x+p^{k-1}$ modulo $\grm$ is $x(x-\bar a_p(f))$ and so, by Hensel's lemma, 
$\bar a_p(f)$ lifts to a root in $\O^\times$.) The action of $G_\Qp$ on the quotient $V_f^- = V_f/V_f^+$ is via $\alpha_f$.
Given any $G_\Q$-stable $\O$-lattice $T\subset V_f$ we let $T^+ = T\cap V_f^+$ and $T^- = T/T^+$. Then $T^+$ is the unique $G_\Qp$-stable free $\O$-summand of rank one on which
$G_\Qp$ acts via $\alpha_f^{-1}\eps^{k-1}$, and $T^-$ is the unique $G_\Qp$-stable free $\O$-module quotient of rank one on which $G_\Qp$ acts via $\alpha_f$.

The following lemma is also well-known, but we also include it for completeness.

\begin{lem} \label{ordlemma} 
Suppose $a_p(f)\in \O^\times$. If $p\mid N$, then $p\mid \mid N$, $k=2$, and $a_p(f)= \pm 1$. 
\end{lem}

\begin{proof} If $f\in S_k(\Gamma_0(N))$ is a newform with trivial Nebentypus such that $p\mid N$,
then $a_p(f) \neq 0$ if and only if $p\mid \mid N$, in which case $a_p(f)^2 = p^{k-2}$. If $a_p(f)\in \O^\times$, then it follows
that $k=2$ and $a_p(f)^2 =1$, so $a_p(f)=\pm 1$. 
\end{proof}

Note that if $f$ is a newform with $p\mid N$ that is ordinary with respect to some embedding $\Q(f)\hookrightarrow \bQp$, then, since $a_p(f)=\pm 1$ by the lemma, it is ordinary with respect to 
all such embeddings. Also, as noted in the proof of the lemma, if $f\in S_2(\Gamma_0(N))$ is a newform with $p\mid \mid N$ then $a_p(f)=\pm 1$ and so $f$ is ordinary with respect to any embedding $\Q(f)\hookrightarrow\bQp$. 

In keeping with the terminology for elliptic curves, we say that a newform $f\in S_2(\Gamma_0(N))$ has 
{\it multiplicative reduction at $p$} if $p\mid \mid N$ and that it has {\it good reduction at $p$} if $p\nmid N$. 
Additionally, we say $f$ has 
{\it split} (resp.~{\it non-split}) multiplicative reduction at $p$ if $p\mid \mid N$ and $a_p(f)=1$ (resp.~$a_p=-1$).
 
\subsection{$\grL$-invariants}\label{Linvariant}
Suppose $f\in S_2(\Gamma_0(N))$ is a newform with split multiplicative reduction at $p$. The Galois representation 
$V_f$ restricted to $G_\Qp$ is an extension
$$
0\rightarrow V_f^+\cong L(1) \rightarrow V_f \rightarrow  V_f^-\cong L \rightarrow 0.
$$
This extension is known to be non-split and semistable but not crystalline\footnote{This is reflected in the compatibility,
proved in \cite{Saito-compat}, of the 
Weil-Deligne representation attached to the dual $V_f^\vee$ by Fontaine with the Weil-Deligne representation
attached by the local Langlands correspondence to the $p$-component $\pi_p$ of the automorphic representation
$\pi=\otimes_{v}\pi_v$ of $\GL_2(\A)$ corresponding to the newform $f$. If $f$ has split (resp.~non-split)
multiplicative reduction at $p$ then $\pi_p$ is the special representation (resp.~the twist of the special representation by 
the unramified quadratic character).}.
Let $\pi_{V_f}:H^1(\Qp,V_f)\rightarrow H^1(\Qp,L)$ be the induced map on cohomology. As the extension is non-split,
the image of $\pi_{V_f}$ is a one-dimensional $L$-space. The $\grL$-invariant $\grL(V_f)$ of $V_f$ is the negative of the `slope' of the 
line $\im(\pi_{V_f})$ with respect to a particular basis of the two-dimensional $L$-space $H^1(\Qp,L)$. 

We have
$$
H^1(\Qp,L) = \Hom_{cts}(G_\Qp,L) = \Hom_{cts}(G_\Qp^{ab,p},L),
$$
where $G_{\Qp}^{ab,p}$ is the maximal abelian pro-$p$ quotient of $G_{\Qp}$. Local class field theory
gives an identification\footnote{To be precise, we normalize the reciprocity law so that uniformizers are taken to arithmetic Frobenius elements.}
$$
\varprojlim_{n} \Q_p^\times/(\Q_p^\times)^{p^n} \isoarrow G^{ab,p}.
$$
From the decomposition $\Q_p^\times = p^\Z \times \Z_p^\times$
we obtain an $L$-basis $\{\psi_\ur$, $\psi_\cyc\}$ of $H^1(\Qp,L)=\Hom_{cts}(G_\Qp^{ab,p},L)$, with 
$$
\psi_\ur(p) = 1 = (\log_p u)^{-1}\cdot \psi_\cyc(u) \ \ \text{and} \ \ \psi_\ur(u)=0 = \psi_\cyc(p).
$$
Recall that $u = \eps(\tilde\gamma)$ is a topological generator of $1+p\Zp$.  The condition that $V_f$ is not crystalline
is equivalent to $\im(\pi_{V_f})\not\subset L\cdot\psi_\ur$. Let $0\neq \lambda \in \im(\pi_{V_f})$ and write 
$\lambda  = x\cdot \psi_\cyc + y \cdot \psi_\ur$. Then $x\neq 0$, and the $\grL$-invariant $\grL(V_f)$ of the extension 
$V_f$ is defined to be 
$$
\grL(V_f) = - x^{-1} y \in L.
$$
This is independent of the choice of $\lambda$.

The non-split extension $V_f$ also defines a line $\ell_{V_f} \in H^1(\Qp,L(1))$ (the image of  the boundary map
$L=H^0(\Qp,L)\rightarrow H^1(\Qp, L(1)$). Under the perfect pairing 
$\langle\cdot,\cdot\rangle:H^1(\Qp, L)\times H^1(\Qp,L(1)) \rightarrow H^2(\Qp,L(1)) = L$
of Tate local duality, the lines $\im(\pi_{V_f})$ and $\ell_{V_f}$ are mutual annihilators. So $\grL(V_f)$
can also be expressed in terms of $\langle\psi_\ur,c\rangle$ and $\langle\psi_\cyc,c\rangle$ for
$0\neq c \in \ell_{V_f}$.

The Kummer isomorphism yields an identification
$$
(\varprojlim_{n} \Q_p^\times/(\Q_p^\times)^{p^n})\otimes_\Zp L \isoarrow H^1(\Qp,L(1)).
$$
Then, together with the above identification of $H^1(\Qp,L)$, the pairing $\langle\cdot,\cdot\rangle$ of 
local Tate duality is identified with the usual $L$-linear pairing 
$$
\Hom_\Zp((\varprojlim_{n} \Q_p^\times/(\Q_p^\times)^{p^n}),L) \times 
(\varprojlim_{n} \Q_p^\times/(\Q_p^\times)^{p^n})\otimes_\Zp L \rightarrow L.
$$
So if $0\neq c \in \ell_{V_f}$, then 
$$
\grL(V_f) = \psi_\ur(c)^{-1}\psi_\cyc(c).
$$
Note that the condition that $V_f$ not be crystalline is equivalent to $\ell_{V_f} \not\subset H^1_f(\Qp,L(1))$,
so $\psi_\ur(c)\neq 0$ as $H^1_f(\Qp,L(1))$
is identified with $(\varprojlim_{n} \Z_p^\times/(\Z_p^\times)^{p^n})\otimes_\Zp L$.

\noindent{\it Example.} Suppose $f$ is associated with an elliptic curve $E/\Q$ with split multiplicative reduction at $p$ and 
let $q_E\in \Q_p^\times$ be the Tate period of $E$. Then $V_f = T_pE\otimes_\Zp\Qp$ is the $G_\Qp$-extension associated
to the image of $q_E$ in $H^1(\Qp,\Qp(1))$ under the Kummer map. That is, $\ell_{V_f} = \Qp\cdot q_E \in 
(\varprojlim_{n} \Q_p^\times/(\Q_p^\times)^{p^n})\otimes_\Zp\Qp$, and so
$\grL(V_f) = \log_p q_E/\ord_p(q_E)$.  As the $j$-invariant $j(q_E) = j(E)\in \Q$ of $E$ is algebraic,
$q_E$ is transcendental
by a theorem of Barr\'e-Sirieix, Diaz, Gramain, and Philibert \cite{BDGP}, and so $\log_p q_E\neq 0$.
Therefore, $\grL(V_f)\neq 0$.

\subsection{Iwasawa--Greenberg Selmer groups}  \label{SelmerSection}
Let $f\in S_k(\Gamma_0(N))$ be a newform that is ordinary with respect to an embedding
$\Q(f)\hookrightarrow \bQp$. Let $L\subset \bQp$ be any finite extension of $\Qp$ 
containing the image of $\Q(f)$ and let $\O$ be the ring of integers of $L$.
 Let $T_f\subset V_f$ be a fixed $G_\Q$-stable $\O$-lattice.

Let $\Lambda_\O = \O[\![\Gamma]\!]$. Let 
$\Psi:G_\Q \twoheadrightarrow \Gamma\subset \Lambda_\O^\times$ be the natural 
projection. This is a continuous $\Lambda_\O$-valued character that is unramified away from
$p$ and totally ramified at $p$. Let $\Lambda_\O^*= \Hom_{cts}(\Lambda_\O,\Qp/\Zp)$ be the Pontryagin dual of $\Lambda_\O$. 
This is a discrete $\Lambda_\O$-module via
$r\cdot \vphi(x) = \vphi(rx)$, for $r,x\in \Lambda_\O$ and $\vphi\in\Lambda_\O^*$. We similarly define a $\Lambda_\O$-module
structure on the Pontryagin dual of any $\Lambda_\O$-module.

Put $\CM = T_f\otimes_\O\Lambda_\O^*$, with $G_\Q$-action given by $\rho_f\otimes\Psi^{-1}$.
Let $\CM^+ = T_f^+\otimes_\O\Lambda_\O^*$ and $\CM^- = \CM/\CM^+$. Let $\Sigma$ be any finite set of primes containing $p$. 
Let $S=\Sigma\cup \{\ell\mid  N\}$. Let $\Q_S$ be the maximal extension of $\Q$ unramified outside $S$ and $\infty$, and 
let $G_{S}=\Gal(\Q_S/\Q)$.  Following Greenberg, we define a Selmer group $\Sel^\Sigma_{\Q_\infty,L}(f)$ by
$$
\Sel^\Sigma_{\Q_\infty,L}(f) = \ker\left\{ H^1(G_S,\CM) \rightarrow 
H^1(I_p, \CM^-)^{G_\Qp}\times \prod_{\ell\in S\backslash\Sigma} H^1(I_\ell, \CM)^{G_\Ql}\right\}.
$$
This is a discrete, cofinite $\Lambda_\O$-module. Its Pontryagin dual
$X_{\Q_\infty,L}^\Sigma(f)$ is a finite $\Lambda_\O$-module. We denote by $Ch_{L}^\Sigma(f)$ the 
$\Lambda_\O$-characteristic ideal of $X_{\Q_\infty,L}^\Sigma(f)$; this is a principal ideal.
In general, these all depend on the choice of $T_f$, but if $\bar\rho_f$ is irreducible, then Lemma
\ref{lattice-lem} shows that
$\Sel^\Sigma_{\Q_\infty,L}(f)$ is independent of $T_f$ up to isomorphism, and hence so is $X^\Sigma_{\Q_\infty,L}(f)$. 
In particular, if $\bar\rho_f$ is irreducible, then the ideal $Ch_L^\Sigma(f)$ does not  depend on the choice of $T_f$.  

Furthermore, if $L_1\supset L$ is a finite extension with ring of integers $\O_1\supset \O$, 
then $T_{f,1} = T_f\otimes_\O\O_1$ is a $G_\Q$-stable $\O_1$-lattice in $V_1 = V_f\otimes_LL_1$
and $T_{f,1}^+ = T_f^+\otimes_\O\O_1$. Hence the Selmer group $\Sel_{\Q_\infty,L_1}^\Sigma(f))$,
defined with respect to the lattice $T_{f,1}$, is canonically isomorphic to $\Sel_{\Q_\infty,L}^\Sigma(f)\otimes_\O\O_1$ as a $\Lambda_{\O_1} = \Lambda_\O\otimes_\O\O_1$-module, from which it follows that its Pontryagin dual $X^\Sigma_{\Q_\infty,L_1}(f)$ is isomorphic to 
$X^\Sigma_{\Q_\infty,L}\otimes_\O\O_1$ as a $\Lambda_{\O_1}$-module and therefore
\begin{equation}\label{Ch-basechange-eq}
Ch_{L_1}^\Sigma(f) = Ch_L^\Sigma(f)\cdot\Lambda_{\O_1}.
\end{equation}

The relation between the Selmer groups $\Sel^{\Sigma_1}_{\Q_\infty,L}(f)$ and 
$\Sel^{\Sigma_2}_{\Q_\infty,L}(f)$ with $\Sigma_1 \subset \Sigma_2$ is clear:
$$
\Sel^{\Sigma_1}_{\Q_\infty,L}(f) = \ker\left\{\Sel^{\Sigma_2}_{\Q_\infty,L}(f) \stackrel{res}{\rightarrow}
\prod_{\ell\in S_2\backslash S_1} H^1(I_\ell,\CM)^{G_\Ql}\right\}.
$$
Each $H^1(I_\ell,\CM)^{G_\Ql}$, $\ell\neq p$, is a cotorsion $\Lambda_\O$-module, and the $\Lambda_\O$-characteristic ideal
of its Pontryagin dual is generated by $P_\ell(\Psi^{-1}\eps^{-1}(\frob_\ell))$, where
$$
P_\ell(X)  = \det(1-\rho_f(\frob_\ell)\mid V_{f,I_\ell})
$$ 
with $V_{f,I_\ell}$ being the space of $I_\ell$-coinvariants of the representation $V_f$. 
In particular, $X^{\Sigma_2}_{\Q_\infty,L}(f)$ is a torsion $\Lambda_\O$-module if and only if
$X^{\Sigma_1}_{\Q_\infty,L}(f)$ is, and 
$$
Ch^{\Sigma_2}_L(f) \supseteq Ch^{\Sigma_1}_L(f)\cdot \prod_{\ell\in\Sigma_2\backslash\Sigma_1} (P_\ell(\Psi^{-1}\eps^{-1}(\frob_\ell)).
$$
Later, we shall see that this last inclusion is often an equality.

If $\Sigma=\{p\}$ then we will omit it from our notation, writing $\Sel_{\Q_\infty,L}(f)$, $X_{\Q_\infty,L}(f)$, and $Ch_L(f)$ instead.

The following lemma shows that if $\Sigma$ is large enough and that if $\bar\rho_f$ is irreducible,
then $\Sel^\Sigma_{\Q_\infty,L}[p^m]$ and $X^\Sigma_{\Q_\infty,L}(f)/p^mX^\Sigma_{\Q_\infty,L}(f)$
depend only on the pair $(T_f/p^mT_f,T_f^+/p^mT_f^+)$ (up to isomorphism).
\begin{lem}\label{Selmer-modpm}
Suppose $\Sigma\supset\{\ell\mid  N\}$ and that $\bar\rho_f$ is irreducible.
Then the inclusion $\CM[p^m]\subset \CM$ induces an identification  
$$
\Sel^\Sigma_{\Q_\infty,L}(f)[p^m] = 
\ker\left\{H^1(G_S,\CM[p^m])\stackrel{res}{\rightarrow} H^1(I_p,\CM^-[p^m])\right\}.
$$
\end{lem}
\noindent That the dependence is only on the pair $(T_f/p^mT_f,T_f^+/p^mT_f^+)$ follows since
$\CM[p^m] \cong T_f/p^mT_f\otimes_\O \Lambda_\O^*[p^m]$,
$\CM^+[p^m]\cong T_f^+/p^mT_f^+\otimes_\O \Lambda_\O^*[p^m]$, and
$\CM^-[p^m] = \CM[p^m]/\CM^+[p^m]$.

\begin{proof}
Since $\bar\rho_f$ is irreducible, the inclusion
$\CM[p^m]\hookrightarrow \CM$ induces an identification 
$H^1(G_\Sigma, \CM[p^m]) = H^1(G_\Sigma,\CM)[p^m]$. So
$\Sel^\Sigma_{\Q_\infty,L}(f)[p^m]$ is the kernel of the restriction map
$H^1(G_\Sigma,\CM[p^m])\rightarrow
H^1(I_p,\CM^-)$, which factors through the restriction map $H^1(G_\Sigma,\CM[p])\rightarrow
H^1(I_p,\CM^-[p^m])$. The kernel 
of the natural map $H^1(I_p,\CM^-[p^m])\rightarrow H^1(I_p,\CM^-)$ is the image of 
$(\CM^-)^{I_p}/p^m(\CM^-)^{I_p}$ via the boundary map. But $(\CM^-)^{I_p} \cong 
\Hom_{cts}(\O,\Qp/\Zp)$ since $I_p$ acts via $\Psi^{-1}$ on $\CM^- \cong \Lambda_\O^*$, and
so $(\CM^-)^{I_p}/p^m(\CM^-)^{I_p}=0$ as $\Hom_{cts}(\O,\Qp/\Zp)$ is $p$-divisible.
\end{proof}

The key to our proofs of both Theorems \ref{thmA} and \ref{thmB} is an understanding of the images
of the restriction maps
\begin{equation}\label{restrict-eq1}
H^1(G_S,\CM) \stackrel{res}{\rightarrow}
H^1(\Qp,\CM^-) \times \prod_{\ell\in S,\ell\neq p} H^1(I_\ell,\CM)^{G_\Ql}
\end{equation}
and
\begin{equation}\label{restrict-eq2}
H^1(G_S,\CM) \stackrel{res}{\rightarrow}
H^1(I_p,\CM^-)^{G_\Qp}\times \prod_{\ell\in S,\ell\neq p} H^1(I_\ell,\CM)^{G_\Ql},
\end{equation}
where $S \supset\{\ell\mid  Np\}$ is any finite set of primes. The kernel of \eqref{restrict-eq2} is, of course,
just $\Sel_{\Q_\infty,L}(f)$. We denote the kernel of \eqref{restrict-eq1} by $\CS$ (it is independent of $S$) 
and let $\CX$ be its Pontryagin dual. As $\CS$ is a submodule of each $\Sel^\Sigma_{\Q_\infty,L}(f)$,
$\CX$ is a quotient of each $X^\Sigma_{\Q_\infty,L}(f)$.

The next two propositions record some properties of the above restriction
maps.
The ideas behind the proofs of these propositions 
are due to Greenberg (see especially \cite[\S\S3,4]{Gr-CIME}, \cite{Gr-sur}, and \cite{Gr-nopseudonull}). As there is not
a convenient reference for the exact case considered here, we have included the details of the arguments.

\begin{prop}\label{Selmerstructureprop1}
Suppose $k\equiv 2 \pmod{p-1}$, $\bar\rho_f$ is irreducible, and $X_{\Q_\infty,L}(f)$ is a torsion $\Lambda_\O$-module.
The restriction maps \eqref{restrict-eq1} and \eqref{restrict-eq2} are surjective.
\end{prop}

\begin{proof} As $H^1(\Qp,\CM^-) \twoheadrightarrow H^1(I_p,\CM^-)^{G_\Qp}$, 
\eqref{restrict-eq2} is surjective if \eqref{restrict-eq1} is. That is, to prove the proposition it suffices to prove surjectivity of
\eqref{restrict-eq1}. To establish this surjectivity we introduce some auxiliary Selmer groups.

Let $\CN = \Hom_\O(T_f,\O(1))\otimes_\O\Lambda_\O$, with $G_\Q$-action 
given by $\eps\rho_f^\vee\otimes\Psi$, and let $\CN^+ = \Hom_\O(T_f/T_f^+,\O(1))\otimes_\O\Lambda_\O$,
which is $G_\Qp$-stable with $G_\Qp$ acting via $\alpha_f^{-1}\eps\otimes\Psi$. These are free $\Lambda_\O$-modules,
and $\CN^+$ is a $\Lambda_\O$-direct summand of $\CN$. Let $\CN^- = \CN/\CN^+$.
The pairing
$$
(\,,\,):\CM \times \CN \rightarrow \Qp/\Zp, \ \ \ 
(t\otimes\vphi,\phi\otimes r) = \vphi(\phi(t)\cdot r),
$$
is a $G_\Q$-equivariant perfect pairing under which $\CM^+$ and $\CN^+$ are mutual annihilators.
Under the induced (perfect) local Tate pairing 
$$
H^i(\Qp,\CM) \otimes H^{2-i}(\Qp,\CN)\rightarrow \Qp/\Zp,
$$
$L_p(\CM)=\im\{H^1(\Qp,\CM^+)\rightarrow H^1(\Qp,\CM)\}$ and 
$L_p(\CN) =\im\{H^1(\Qp,\CN^+)\rightarrow H^1(\Qp,N)\}$ are also mutual annihilators. 
Let
$$
\Sel^S(\CN) = \ker\left\{H^1(G_S, \CN) \stackrel{res}{\rightarrow} H^1(\Qp,\CN)/L_p(\CN)\hookrightarrow H^1(\Qp,\CN^-)\right\}.
$$
Let $\Sha^1(\Q,S,\CN) \subseteq \Sel^\Sigma(\CN)$ consist of those classes that are trivial at all places in $S$. 

For $\ell\neq p$, $H^1(\F_\ell,\CM^{I_\ell})=0$ and so $H^1(\Q_\ell,\CM)\isoarrow H^1(I_\ell,\CM)^{G_\Ql}$. 
Also, $H^2(\Qp,\CM^+)=0$ as its dual is $H^0(\Qp,\CN^-)=0$, so $H^1(\Qp,\CM)/L_p(\CM)\isoarrow H^1(\Qp,\CM^-)$.
Global Tate duality then identifies the dual of the cokernel of \eqref{restrict-eq1} with 
$\Sel^S(\CN)/\Sha^1(\Q,S,\CN)$ (cf.~\cite[Prop.~3.1]{Gr-sur}).  
As $\bar\rho_f$ is irreducible, $H^1(G_S, \CN)$ is $\Lambda_\O$-torsion-free, and hence
so are $\Sel^S(\CN)$ and $\Sha^1(\Q,S,\CN)$. Therefore, to prove the desired surjectivity it suffices to show that
$\Sel^S(\CN)$ is a torsion $\Lambda_\O$-module (and so trivial). We prove that $\Sel^S(\CN)$ is torsion by 
exhibiting elements $x$ in the maximal ideal of $\Lambda_\O$ such that $\Sel^S(\CN)/x\Sel^S(\CN)$ has finite order.

Let $x = \gamma - u^m\in \Lambda_\O$ with $m$ an integer. Let $N_x=\CN/x\CN$, $N^+_x = \CN^+/x\CN^+$,
and $N_x^- = N_x/N_x^+$.  These are free $\O$-modules. If $p\nmid N$ or $m\neq 0$, then the natural injection 
$$
H^1(G_S,\CN)/xH^1(G_S,\CN) \hookrightarrow H^1(G_S,N_x)
$$ 
induces an injection
\begin{equation}\label{Selinjecteq1}
\Sel^S(\CN)/x\Sel^S(\CN)\hookrightarrow \Sel^S(N_x) = \ker\left\{H^1(G_S, N_x) \rightarrow H^1(\Qp,N_x^-)\right\}.
\end{equation}
For this, we first note that the image of the induced map from $\Sel^S(\CN)/x\Sel^S(\CN)$ to 
$H^1(G_S,N_x)$ lies in $\Sel^S(N_x)$. It remains to prove injectivity. Let $c\in \Sel^S(\CN)$ be such that it has
trivial image in $\Sel^S(N_x)$. Then $c = xd$ for some 
$d\in H^1(G_S,\CN)$ such that $xd = 0$ in $H^1(\Qp,\CN^-)$. The kernel of multiplication by $x$ on $H^1(\Qp,\CN^-)$
is the image of $H^0(\Qp,N_x^-)$.  But $N_x^-$ is a free $\O$-module with $G_\Qp$ acting via the character $\alpha_f\eps^{2-k+m}\omega^{-m}$, and so $H^0(\Qp,N_x^+)=0$ unless $m=k-2$ and $\alpha_f=1$. But $\alpha_f=1$ only if $p\mid \mid N$ and $k=2$.
It follows that that if $p\nmid N$ or $m\neq 0$, then multiplication by $x$ is injective on $H^1(\Qp,\CN^-)$ and,
therefore, $d\in \Sel^S(\CN)$,
proving the injectivity in \eqref{Selinjecteq1}.

From \eqref{Selinjecteq1} it follows that to prove $\Sel^S(\CN)$ is torsion it suffices to show that there is some $m\neq 0$ such that 
$\Sel^S(N_x)$ has finite order.  As $\Sel^S(N_x)$ has finite order if and only if $\Sel^S(N_x)\otimes_\Zp \Qp/\Zp$
has finite order - in which case it must be trivial -  it suffices to prove the latter.  
Furthermore, as $\bar\rho_f$ is irreducible and so $H^1(G_S,N_x)$ - 
and hence also $\Sel^S(N_x)$ - is 
a torsion-free $\O$-module and therefore free, it would then follow that $\Sel^S(N_x) = 0$.

Let $M_x = N_x\otimes_\Zp\Qp/\Zp$ and $M_x^ - = N_x^- \otimes_\Zp \Qp/\Zp$.  From the injections
$H^1(G_S,N_x)\otimes_\Zp \Qp/\Zp\hookrightarrow H^1(G_S M_x)$ and $H^1(\Qp,N_x^-)\otimes_\Zp \Qp/\Zp 
\hookrightarrow H^1(\Qp,M_x^-)$ we deduce an injection
\begin{equation}\label{Selinjecteq2}
\Sel^S(N_x)\otimes_\Zp \Qp/\Zp \hookrightarrow \Sel^S(M_x) = \ker\left\{ H^1(G_S,M_x) 
\stackrel{res}{\rightarrow} H^1(\Qp,M_x^-)\right\}.
\end{equation}
So to prove that there is an $m\neq 0$ such that $\Sel^S(N_x)$ has finite order, it suffices to find such an $m$ for which
$\Sel^S(M_x)$ has finite order.

Let $m\neq 0$ be an integer such that $m\equiv 0 \pmod{p-1}$. 
Let $ y = \gamma - u^{k-2-m}$. Then, as $k\equiv 2\pmod{p-1}$, 
$M_x \cong \CM[y]$ as $\O[G_\Q]$-modules, and the isomorphism can be chosen so that $M_x^-$ is identified with 
$\CM^-[y]$. It follows that 
\begin{equation}\label{Selinjecteq3}
\Sel^S(M_x) = \Sel^S(\CM[y]) \hookrightarrow \Sel^S_{\Q_\infty,L}(f)[y],
\end{equation}
where $\Sel^S(\CM[y])$ is defined just as $\Sel^S(M_x)$, and where
the injection is induced by the natural identification $H^1(G_S, \CM[y])\isoarrow H^1(G_S, \CM)[y]$
(which is injective as $\bar\rho_f$ is irreducible). 

As $X_{\Q_\infty,L}(f)$ is a torsion $\Lambda_\O$-module so is $X^S_{\Q_\infty,L}(f)$. Therefore, for all but finitely many integers $m$, 
$X^S_{\Q_\infty,L}(f)/yX^S_{\Q_\infty,L}(f)$
has finite order. As the latter is dual to $\Sel^S_{\Q_\infty,L}(f)[y]$, it follows from \eqref{Selinjecteq3} that there is an $m\neq 0$
with $m\equiv 0\pmod{p-1}$
such that $\Sel^S(M_x)$ has finite order. As explained above, the existence of such an $x$ implies the desired surjectivity of 
\eqref{restrict-eq1}.
\end{proof}

\begin{prop}\label{Selmerstructureprop2}
Suppose $k\equiv 2 \pmod{p-1}$, $\bar\rho_f$ is irreducible, and $X_{\Q_\infty,L}(f)$ is a torsion $\Lambda_\O$-module.
\begin{itemize}
\item[\rm (i)] $\CX$ has no non-zero finite-order $\Lambda_\O$-submodules.
\item[\rm (ii)] Let $\Sigma$ be any finite set of primes containing $p$. Then 
$X^\Sigma_{\Q_\infty,L}(f)$ has no non-zero finite-order $\Lambda_\O$-submodules.
\end{itemize}
\end{prop}

\begin{proof}
To prove part (i), let $S\supset\{\ell\mid Np\}$ be any finite set of primes and let
$$
\CP_S= H^1(\Qp,\CM^-)\times \prod_{\ell\in S,\ell\neq p}H^1(\Ql,\CM).
$$
For $x = \gamma-u^{m}\in\Lambda_\O$, $\CP_S[x]$ is a quotient of 
$$
P_{S,x} = H^1(\Qp,\CM[x])/L_p(\CM[x]) \times \prod_{\ell\in S,\ell\neq p}H^1(\Ql,\CM[x]),
$$
where $L_p(\CM[x]) = \im\{H^1(\Qp,\CM^+[x])\rightarrow H^1(\Qp,\CM[x])\}$.
Therefore the cokernel of the restriction map $H^1(G_S,\CM[x])=H^1(G_S,\CM)[x]\rightarrow \CP_S[x]$
is a quotient of the cokernel of the restriction map $H^1(G_S,\CM[x])\rightarrow P_{S,x}$. By
global Tate duality, the Pontryagin dual of the latter is a subquotient of $\Sel^S(N_x)$, where
$N_x$ and $\Sel^S(N_x)$ are as in \eqref{Selinjecteq1}. But, as
shown in the proof of Proposition \ref{Selmerstructureprop1}, 
$m$ can be chosen so that $\Sel^S(N_x)=0$ and hence so that
$H^1(G_S,\CM)[x]\twoheadrightarrow\CP_S[x]$.  
It then follows from an application of the 
snake lemma to multiplication by $x$ of the short exact sequence 
$$
0\rightarrow \CS \rightarrow H^1(G_S,\CM) \rightarrow \CP_S \rightarrow 0
$$
that,
for such a choice of $m$,
\begin{equation}\label{Selinjecteq4}
\CS/x\CS \hookrightarrow H^1(G_S,\CM)/xH^1(G_S,\CM).
\end{equation}
However, as shown in both \cite[Lem.~3.3.18]{SU-MCGL} and \cite[Prop.~2.6.1]{Gr-nopseudonull}, the right-hand side of \eqref{Selinjecteq4} 
is trivial for all but finitely many $m$,
so the $m$ can also be chosen so that  $\CS/x\CS = 0$. 
Let $X\subseteq \CX$ be a sub-$\Lambda_\O$-module of finite order, and let $X^*$ be
its Pontryagin dual.  Then $X^*/xX^*$ is a quotient
of $\CS/x\CS$ and so is $0$. By Nakayama's lemma $X^*=0$, hence $X=0$.
This proves (i). 

To prove part (ii), let $S\supset \Sigma\cup\{\ell\mid Np\}$ and let 
$$
\CP_{S,\Sigma} = H^1(I_p,\CM^-)^{G_\Qp}\times \prod_{\ell\in S\backslash \Sigma}H^1(\Ql,\CM)
$$
and
$$
\CP_{S,\Sigma,x} = H^1(\Qp,\CM[x])/L_p(\CM[x]) \times \prod_{\ell\in S\backslash\Sigma}H^1(\Ql,\CM[x]).
$$
We may then argue as in the proof of part (i) but with 
$\CP_S$ replaced by $\CP_{S,\Sigma}$.
Then $\CS$ is replaced by $\Sel^\Sigma_{\Q_\infty,L}(f)$. Furthermore, as $\CP_{S,\Sigma,x}$
is a quotient of $\CP_{S,x}$, the surjectivity of the restriction map $H^1(G_S,\CM[x])\rightarrow\CP_{S,\Sigma,x}$,
and hence of the restriction map $H^1(G_S,\CM[x])\rightarrow \CP_{S,\Sigma}[x]$,
follows for a suitable $x=\gamma-u^m\in\Lambda_\O$ from the surjectivity of the restriction map onto $\CP_{S,x}$ 
established in the proof of part (i).
\end{proof}

Let $F_L^\Sigma(f)$ be the $\Lambda_\O$-Fitting ideal of $X^\Sigma_{\Q_\infty,L}(f)$.
The following is a straight-forward consequence of the preceding propositions.

\begin{lem}\label{Ch-Fitt-coro}
Suppose  $k\equiv 2 \pmod{p-1}$ and $\bar\rho_f$ is irreducible. 
\begin{itemize}
\item[\rm (i)] $Ch_L^\Sigma(f) = Ch_L(f)\cdot(\prod_{\ell\in \Sigma, \ell\neq p} 
P_\ell(\Psi^{-1}\eps^{-1}(\frob_\ell))$.
\item[\rm (ii)] $F^\Sigma_L(f) = Ch_L^\Sigma(f)$.
\end{itemize}
\end{lem}

\begin{proof}
If $X^\Sigma_{\Q_\infty,L}(f)$ is not a torsion $\Lambda_\O$-module (equivalently, $X_{\Q_\infty,L}(f)$ is not a torsion
$\Lambda_\O$-module), then $Ch_L(f)$, $Ch^\Sigma_L(f)$, and $F_L^\Sigma(f)$ are all zero, so there is 
nothing to prove.  We suppose then that $X^\Sigma_{\Q_\infty,L}(f)$ is a torsion $\Lambda_\O$-module.

Part (i) is immediate from Proposition \ref{Selmerstructureprop1} and the definition of characteristic
ideals. For part (ii), we first note that $F_L^\Sigma(f) \subset Ch_L^\Sigma(f)$. Let $\gra$ be the kernel 
of the quotient $\Lambda_\O/F^\Sigma_L(f)\twoheadrightarrow
\Lambda_\O/Ch^\Sigma_L(f)$. Since $X^\Sigma_{\Q_\infty,L}(f)$ is a torsion $\Lambda_\O$-module
and $Ch_L^\Sigma(f)$ is a principal ideal,
there exists $\lambda = \gamma-u^m\in \Lambda_\O$ such that 
$\lambda$ is not a zero-divisor in $\Lambda_\O/Ch_L^\Sigma(f)$ and 
$X^\Sigma_{\Q_\infty,L}(f)/\lambda X^\Sigma_{\Q_\infty,L}(f)$ is a torsion $\Lambda_\O/\lambda\Lambda_\O = \O$-module. 
The size of this module is then equal to the size of both $\Lambda_\O/(\lambda,F^\Sigma_L(f))$ and 
$\Lambda_\O/(\lambda,Ch_L^\Sigma(f))$ (which are necessarily finite), the first by basic
properties of Fitting ideals and the second by Proposition \ref{Selmerstructureprop2}(ii) and a 
standard argument\footnote{The argument: A finitely-generated torsion $\Lambda_\O$-algebra
$X$ admits a $\Lambda_\O$-homomorphism $X\rightarrow Y=\prod_{i=1}^r \Lambda_\O/(f_i)$ with finite-order kernel 
$\gra$ and cokernel $\grb$
and such that the $\Lambda_\O$-characteristic ideal of $X$ is $(f_1\cdots f_r)$. Let $f=f_1\cdots f_r$. 
If $X$ has no finite-order $\Lambda_\O$-submodules, then the map to $Y$ is an injection. 
Multiplying the short exact sequence 
$0 \rightarrow X\rightarrow Y \rightarrow \grb\rightarrow 0$ by $\lambda=\gamma-u^m$ and applying the snake lemma is easily seen to give
$$
\#X/\lambda X = \#Y/\lambda Y = \prod\#\Lambda_\O/(\lambda,f_i) = \prod\#\O/(f_i(u^m-1)) = \#\O/(f(u^m-1)) = \#\Lambda_\O/(\lambda,f),
$$
where we have written $f_i(u^m-1)$ and $f(u^m-1)$ for the respective images of $f_i$ and $f$ under the 
continuous $\O$-algebra homomorphism $\Lambda_\O\rightarrow\O$ sending
$\gamma$ to $u^m$.}
from Iwasawa theory. If follows that 
the natural projection $\Lambda_\O/(\lambda,F_L^\Sigma(f))\twoheadrightarrow \Lambda_\O/(\lambda,Ch_L^\Sigma(f))$
is an isomorphism.  Applying the snake lemma to the diagram obtained by multiplying the short exact sequence 
$$
0\rightarrow \gra \rightarrow \Lambda_\O/F_L^\Sigma(f) \rightarrow \Lambda_\O/Ch_L^\Sigma(f) \rightarrow 0
$$
by $\lambda$ then yields an exact sequence
$$
0\rightarrow \gra/\lambda\gra \rightarrow \Lambda_\O/(\lambda,F_L^\Sigma(f)) \isoarrow 
\Lambda_\O/(\lambda,Ch_L^\Sigma(f))\rightarrow 0.
$$
Therefore $\gra/\lambda\gra$, and hence $\gra$, is $0$.
\end{proof}

\subsection{$p$-adic $L$-functions}\label{padicLs}
Let $f$, $L$, $\O$, and $\Lambda_\O$ be as in the preceding section, with the assumption that
$k\geq 2$ and $f$ is ordinary with respect to $L$. Amice and V\'elu \cite{AV} and Vishik \cite{Vishik} 
(see also \cite{MTT})  constructed a $p$-adic $L$-function for $f$. This is a power series 
$\CL_f\in \Lambda_\O$ with the property that if $\phi:\Lambda_\O\rightarrow\bQp$ is 
a continuous $\O$-homomorphism such that $\phi(\gamma) = \zeta u^m$ with $\zeta$ a 
primitive $p^{t_\phi-1}$th root of unity and $0\leq m \leq k-2$ an integer, then\footnote{The
power of $-2\pi i$ in the denominator of this formula is incorrectly given as $(-2\pi i)^m$ 
in some of the formulas in \cite{SU-MCGL}, namely in the introduction, in \S3.4.4, and in Theorem 3.26 of {\it loc.~cit.} 
In these cases the correct factor is $(-2\pi i)^{m+1}$. This error originates in the difference between
$\Omega_{f}^\pm$ as defined in \cite[\S3.3.3]{SU-MCGL} and the $\Omega^\pm$ in \cite[I.9]{MTT}: 
$\Omega^\pm = -2\pi i \Omega_f^\pm$. The exponents of $-2\pi i$ are correct in the formulas in \cite{SU-MCGL}
for the $L$-function of $f$ twisted by a Hecke character of the imaginary quadratic field $\mathcal{K}$.}
\begin{equation}\label{padicLeq1}
\begin{split}
\CL_f(\phi):=\phi(\CL_f) = e(\phi)\frac{p^{t_\phi'(m+1)}m!L(f,\chi_\phi^{-1}\omega^{-m},m+1)}
{ (-2\pi i)^{m+1} G(\chi_\phi^{-1}\omega^{-m}) \Omega_f^{\mathrm{sgn}((-1)^m)}},  \\
e(\phi) = \alpha_p^{-t_\phi}\left(1 - \frac{\omega^{-m}\chi_\phi^{-1}p^{k-2-m}}{\alpha_p}\right)
\left( 1 - \frac{\omega^m\chi_\phi(p) p^m}{\alpha_p}\right),
\end{split}
\end{equation}
where $\alpha_p$ is the unique (unit) root  in $\O^\times$ of 
$x^2-a_p(f)x+p^{k-1}$ if $p\nmid N$ and $\alpha_p = a_p(f)$ if $p\mid N$, 
$t_\phi' = 0$ if $t_\phi=1$ and $p-1\mid m$ and otherwise $t_\phi' = t_\phi$,
$\chi_\phi$ is the primitive Dirichlet character of $p$-power order and conductor (which can
be viewed as a finite-order character of $\Z_p^\times$) such that
$\chi_\phi(u) = \zeta^{-1}$, $G(\chi_\phi^{-1}\omega^{-m})$ is the usual Gauss sum (and so 
equals $1$ if $t_\phi'=0$), and $\Omega_f^\pm$ are the canonical periods of $f$ (these are 
well-defined up to a unit in $\O$; see \cite[\S3.3.3]{SU-MCGL}).

Let $\Sigma$ be a finite set of primes. We define an incomplete $p$-adic $L$-function
$\CL^\Sigma_f \in \Lambda_\O$ by
\begin{equation}\label{padicLeq2}
\CL^\Sigma_f = \CL_f \cdot \prod_{\ell\in \Sigma,\ell\neq p} 
P_\ell(\Psi^{-1}\eps^{-1}(\frob_\ell)).
\end{equation}
Note that
\begin{equation*}
  P_\ell(\Psi^{-1}\eps^{-1}(\frob_\ell)) =
\begin{cases} 1-a_\ell(f)\ell^{-1}\Psi^{-1}(\frob_\ell) + \ell^{k-3}\Psi^{-2}(\frob_\ell)
& \ell\nmid N \\
1-a_\ell(f)\ell^{-1}\Psi^{-1}(\frob_\ell)  & \ell\mid N.
\end{cases}
\end{equation*}
In particular, the value of $\CL_f^\Sigma$ under a continuous $\O$-algebra homomorphism
$\phi:\Lambda_\O\rightarrow \bQp$ such that $\phi(\gamma) = \zeta u^m$, $0\leq m\leq k-2$,
can be expressed in terms of a special value of an incomplete $L$-function:
$$
\CL_f^\Sigma(\phi) = e(\phi)\frac{p^{t_\phi'(m+1)}m!L^{\Sigma\backslash\{p\}}(f,\chi_\phi^{-1}\omega^{-m},m+1)}
{ (-2\pi i)^{m+1} G(\chi_\phi^{-1}\omega^{-m}) \Omega_f^{\mathrm{sgn}((-1)^m)}}.
$$

\begin{rmk} Let $\Z(f)$ be the ring of integers of $\Q(f)$ and let $\grp$ be the prime of $\Z(f)$
determined by the chosen embedding $\Q(f)\hookrightarrow \bQp$. Then $\Omega_f^\pm$ is well-defined up to a unit in the localization $\Z(f)_{(\grp)}$ of $\Z(f)$, and the value of the
$p$-adic $L$-function under a homomorphism $\phi$ as above lies in a finite extension of $\Z(f)_{(\grp)}$. It is in this way that 
period-normalized values of the $L$-function $L(f,s)$ and its twists,
which {\it a priori} are complex values, can be viewed as being in $\bQp$ without fixing an 
isomorphism $\bQp\isom\C$.
\end{rmk}

Suppose $f$ has split multiplicative reduction at $p$. Then it follows easily from \eqref{padicLeq1} that if
 $\phi_0:\Lambda_\O\rightarrow \bQp$ is the $\O$-algebra homomorphism such that $\phi_0(\gamma) = 1$,
 then $\CL_f(\phi_0) = 0$. In particular, $\CL_f = (\gamma-1)\cdot \CL_f'$ for some $\CL_f'\in \Lambda_\O$.
Greenberg and Stevens \cite[Thm.~7.1]{GrSt} proved that $\CL_f'(\phi_0) = \phi_0(\CL_f')$ is
related to the $\grL$-invariant of $V_f$ by the formula
\begin{equation}\label{GrSteq}
\CL'_f(\phi_0) = (\log_p u)^{-1}\grL(V_f)\frac{L(f,1)}{-2\pi i \Omega_f^+}.
\end{equation}
 More precisely, 
if we identify $\Lambda_\O$ with the power-series ring $\O[\![T]\!]$ by sending $\gamma$ to $1+T$,
and if we let $L_p(f,s) = \CL_f(u^{s-1}-1)$, $s\in\Zp$, then Greenberg and Stevens proved that
$$
\frac{d}{ds}L_p(f,s) \mid_{s=1}  = \grL(V_f)\frac{L(f,1)}{-2\pi i \Omega_f^+}.
$$
This is easily seen to be equivalent to \eqref{GrSteq}. This formula was conjectured by Mazur, Tate, and Teitelbaum \cite[\S13]{MTT}.

\subsection{The Iwasawa--Greenberg Main Conjecture}\label{MCsection}

Let $f$, $L$, $\O$, $\Lambda_\O$, $\CL_f$, etc., be as in the preceding sections.
Along the lines of Iwasawa's original Main Conjecture for totally real number fields, Mazur and
Swinnerton-Dyer (for modular elliptic curves) and Greenberg (more generally) made the following
conjecture.
\begin{conj} If $\Sigma$ is any finite
set of primes containing $p$, then $X^\Sigma_{\Q_\infty,L}(f)$ is a torsion $\Lambda_\O$-module and $Ch_L^\Sigma(f) = (\CL_f^\Sigma)$ in $\Lambda_\O\otimes_\Zp\Qp$ and even in $\Lambda_\O$ if $\bar\rho_f$ is irreducible.
\end{conj}
\noindent It follows easily from Lemma \ref{Ch-Fitt-coro}(i) and \eqref{padicLeq2} that if this  conjecture holds for one set $\Sigma$ 
then it holds for all sets $\Sigma$.  Also, the conjecture with $L$ replaced by any finite extension implies the conjecture for $L$, as can be seen by the observations in 
Section \ref{SelmerSection} on the relation \eqref{Ch-basechange-eq} between $Ch^\Sigma_L(f)$ and $Ch^\Sigma_{L_1}(f)$
for a finite extension $L_1\supset L$.

In \cite{SU-MCGL} the following theorem was proved, in combination with results of Kato \cite{Kato}, which established this conjecture for a large class of modular forms.

\begin{thm}\label{MCthm1} Suppose
\begin{itemize}
\item[\rm (i)] $k\equiv 2\pmod{p-1}$;
\item[\rm (ii)] $\bar\rho_f$ is irreducible;
\item[\rm (iii)] there exists a prime $q\neq p$ such that $q\mid\mid N$ and $\bar\rho_f$ is ramified at $q$;
\item[\rm (iv)] $p\nmid N$ (this is automatic if $k\neq 2$).
\end{itemize}
Then for any finite set of primes $\Sigma$, $X^\Sigma_{\Q_\infty,L}(f)$ is a torsion $\Lambda_\O$-module and $Ch^\Sigma_L(f) = (\CL_f^\Sigma)$ in $\Lambda_\O$.
\end{thm}

In \cite{SU-MCGL} an additional hypothesis is required to conclude equality in $\Lambda_\O$ and not just in $\Lambda_\O\otimes_\Zp\Qp$:
\begin{equation*}\tag{*}
 \text{there exists an $\O$-basis of $T_f$ such that the image of $\rho_f$ contains $\SL_2(\Zp)$.}
\end{equation*}
This hypothesis was included because it is part of the statement of \cite[Thm.~17.4]{Kato}. However, a 
closer reading of the proof of {\it loc.~cit.}~shows that all that is necessary is that (a) $\bar\rho_f$ be irreducible and (b) there exist an element $g\in \Gal(\bQ/\Q[\mu_{p^\infty}])$ such 
that $T_f/(\rho_f(g)-1)T_f$ is a free $\O$-module of rank one, as we explain in the following paragraph. All references to theorems or sections in the following paragraph are to \cite{Kato} unless 
otherwise indicated.

Hypothesis (*) intervenes in the proof of Theorem 17.4 through Theorem 15.5(4), which is proved
in \S 13.14. Hypothesis (a) together with Lemma \ref{lattice-lem} of this paper implies that, in the notation of \cite{Kato}, the conclusion
in \S 13.14 that $T_f = a\cdot V_{O_\lambda}(f)$ for some $a\in F_\lambda^\times$ holds; Lemma \ref{lattice-lem} of this paper can replace the reference 
to Lemma 14.7 in \S 13.14, which is the only explicit use of a basis with an image containing $\SL_2(\Zp)$ in the proof of Theorem 15.5(4).
Hypothesis (a) also, of course, ensures that the hypotheses of Theorem 12.4(3) hold, as needed in \S 13.14.
Hypothesis (b) ensures that the hypotheses of Theorem 13.4(3) hold. The proof of Theorem 15.5(4) in \S 13.14 then holds with (*) replaced by the hypotheses (a) and (b) above.

We now check that (a) and (b) hold under the hypotheses of Theorem \ref{MCthm1}.
Hypothesis (a) is just hypothesis (ii) of the theorem.
Hypothesis (b) is satisfied in light of hypothesis (iii) of the theorem: As $q\mid\mid N$, the action of
$I_q$ on $V_f$ is nontrivial and unipotent and in particular factors through the tame quotient (this is a consequence of the `local-global' compatibility of the Galois representation $\rho_f$ \cite[Thm.~A]{Carayol}). 
It follows that $\rho_f(\tau)$ is unipotent for any $\tau\in I_q$ projecting to a topological generator of the tame quotient and, since $\bar\rho_f$ is ramified at $q$, $\bar\rho_f(\tau)\neq 1$,
hence $T_f/(\rho_f(\tau)-1)T_f$ is a free $\O$-module of rank one. As $\tau\in \Gal(\bQ/\Q[\mu_{p^\infty}])$, condition (b) holds for $g=\tau$.

We also take this opportunity to note that the reference to \cite{Vatsal} in the proof of \cite[Prop.~12.3.6]{SU-MCGL} is not sufficient. 
It may be that the weight two specialization of the Hida family in {\it loc.~cit.}~that has trivial character also has multiplicative reduction
at $p$. This case is excluded in \cite{Vatsal}, though the ideas in that paper can be extended to this case, as is explained in 
\cite{Chida-Hsieh}. The reference to \cite[Thm.~1.1] {Vatsal} must be augmented by a reference to \cite[Thm.~C]{Chida-Hsieh}.

The purpose of this paper is, of course, to show that hypothesis (iv) of Theorem \ref{MCthm1} can be removed.

The main results of \cite{SU-MCGL} show that for a suitable imaginary quadratic field $K$ and a large enough set $\Sigma$,
the equality $Ch^\Sigma_{L}(f)Ch^\Sigma_{L}(f\otimes\chi_K) = (\CL_f^\Sigma\CL_{f\otimes\chi_K}^\Sigma)$ holds, where
$f\otimes\chi_K$ is the newform associated with the twist of $f$ by the primitive quadratic Dirichlet character corresponding to $K$.
When $p\nmid N$, this equality can be refined to an equality of the individual factors via the inclusions 
$\CL_f^\Sigma\in Ch_L^\Sigma(f)$ and $\CL_{f\otimes\chi_K}^\Sigma\in Ch_L^\Sigma(f\otimes\chi_K)$, which are proved
in \cite{Kato}. When $p\mid N$, these inclusions do not follow directly from \cite{Kato}; additional arguments are required.

\subsection{Hida families}\label{HidaThy}
Let $f\in S_k(\Gamma_0(N))$ be a newform that is ordinary with respect to an embedding
$\Q(f)\hookrightarrow \bQp$. Write $N=p^rM$ with $p\nmid M$ (so $r=0$ or $1$ 
by Lemma \ref{ordlemma}).  Let $L\subset \bQp$ be any finite extension of $\Qp$ 
containing the image of $\Q(f)$ and let $\O$ be the ring of integers of $L$. Let
$R_0 = \O[\![X]\!]$.  Hida proved that there is a finite, local $R_0$-domain $R$ and a formal 
$q$-expansion 
$$
\ff = \sum_{n=1}^\infty a_n q^n \in R[\![q]\!], \ \ \ a_1 = 1,
$$
satisfying:
\begin{itemize}
\item $R = R_0[\{a_\ell\ : \ \ell = \text{prime}\}]$;
\item if $\phi:R\rightarrow \bQp$ is a continuous $\O$-algebra homomorphism such 
that $\phi(1+X) = (1+p)^{k'}$ with $k'>2$ and $k'\equiv k \pmod{p-1}$, then 
$\sum_{n=1}^\infty \phi(a_n)q^n$ is the $q$-expansion of a $p$-stabilized newform,
in the sense that there is a newform $f_\phi\in S_{k'}(\Gamma_0(M))$ and an 
embedding $\Q(f_\phi)\hookrightarrow \bQp$ such that $\phi(a_\ell) = a_\ell(f_\phi)$
for all primes $\ell\neq p$ and $\phi(a_p)$ is the unit root of the polynomial
$x^2 - a_p(f_\phi)x + p^{k'-1}$;
\item there is a continuous $\O$-algebra homomorphism $\phi_0:R\rightarrow \O$ such that $\phi_0(1+X) = (1+p)^k$
and $\phi_0(a_\ell) = a_\ell(f)$, $\ell\neq p$, and $\phi_0(a_p)$ is the unit root of 
$x^2-a_p(f)x+p^{k-1}$ if $r=0$ and $\phi_0(a_p) = a_p(f)$ if $r=1$.
\end{itemize}
Furthermore, after possibly replacing $L$ with a finite extension, we may assume
\begin{itemize}
\item $\O$ is integrally closed in $R$.
\end{itemize}
Then, as explained by Greenberg and Stevens \cite{GrSt} (see also \cite[(1.4.7)]{Nek}), 
\begin{itemize}
\item there is an integer $c$ and an $\O$-algebra embedding 
$$
R \hookrightarrow R_c = \left\{ \sum_{i=0}^\infty u_i (x-k)^i \ : \ u_i\in L, \ 
\lim_{i\rightarrow \infty} \ord_p(u_i) + ci = +\infty \right\} \subset L[\![x]\!] 
$$
such that the induced embedding of $R_0$ sends $1+X$ to the power series
expansion of $(1+p)^x$ about $x=k$ and $\phi_0$ is the homomorphism induced
by evaluating at $x=k$.
\end{itemize}
Then evaluating at $x=k'$ for an integer $k'>2$ with $k'\equiv k \pmod{(p-1)p^c}$
defines a continuous $\O$-algebra homomorphism $\phi_{k'}:R\rightarrow L$
such that $\phi_k(1+X)=(1+p)^{k'}$ with corresponding newform 
$f_{\phi_{k'}}\in S_{k'}(\Gamma_0(M))$. Furthermore, it is clear that given any integer $m>0$,
there is an integer $r_m>0$ such that if $k' \equiv k \pmod{(p-1)p^{r_m}}$, then 
$\phi_{k'} \equiv \phi_0 \pmod{p^m\O}$; in particular, for all primes $\ell\neq p$
$$
a_\ell(f_{\phi_{k'}}) \equiv a_\ell(f) \pmod{p^m\O}.
$$
For each integer $m$ we choose such a $k'=k_m$ and write $f_m$ for the corresponding
$f_{\phi_{k_m}}$.  Note that we have chosen $k_m>2$ so that $f_m$ is a newform of level 
not divisible by $p$, though $p$ might divide the level of $f$.

Suppose that $\bar\rho_f$ is irreducible. Then there is a free rank two $R$-module $\T$ and a
continuous Galois representation 
$$
\rho_R:G_\Q\rightarrow \Aut_R(\T)
$$
that is unramified at each $\ell\nmid pN$ and such that for any such prime
$\trace\rho_R(\frob_\ell) = a_\ell \in R$.  In particular for $\phi:R\rightarrow \O$ being
$\phi_0$ or one of the homomorphisms $\phi_{k_m}$, $T_{f_\phi} =\T\otimes_{R,\phi}\O$ 
is a $G_\Q$-stable $\O$-lattice in $\T\otimes_{R,\phi}L \cong V_{f_\phi}$.
Let $T_f = T_{f_{\phi_0}}$ and $T_{f_m} = T_{f_{\phi_{k_m}}}$. 
Since $\phi_0$ and $\phi_m$ agree modulo $p^m$,
reduction modulo $p^m$ induces 
identifications
\begin{equation}\label{latticeequality}
T_f/p^m T_f  = \T\otimes_{R,\phi_0} \O/p^m\O = \T\otimes_{R,\phi_m}\O/p^m\O = T_{f_m}/p^m T_{f_m}
\end{equation}
as $\O[G_\Q]$-modules.

Suppose also that 
$$
\alpha_f^{-1}\eps^{k-1}\not\equiv \alpha_f \pmod{\grm}.
$$ 
Then there is a free rank-one $G_\Qp$-stable $R$-summand $\T^+\subset \T$ such that for
any of the $\phi$ as before, $\T^+\otimes_{R,\phi}\O  = T_{f_\phi}^+$. The identification
$T_f/p^m T_f = T_{f_m}/p^m T_{f_m}$ induces an identification 
\begin{equation}\label{latticeequality2}
T_f^+/p^m T_f^+  = T_{f_m}^+/p^m T_{f_m}^+.
\end{equation}

Greenberg and Stevens \cite{GrSt} and others have shown that the $p$-adic $L$-functions
$\CL_{f_\phi}$ for the forms $f_\phi$ arising from a Hida family fit into a `two-variable' $p$-adic $L$-function. In particular, following Emerton, Pollack, and Weston, we have the following.

\begin{prop}\label{2variableLprop}$($\cite[\S3 esp.~Prop.~3.4.3]{EPW}$)$ 
Let $\Sigma$ be a finite set of primes containing $p$.
If $\bar\rho_f$ is irreducible, then there exists
$\CL^\Sigma_\ff \in R[\![\Gamma]\!]$ such that for each continuous $\O$-algebra homomorphism
$\phi:R \rightarrow \bQp$ as above, the image of $\CL^\Sigma_\ff$ in 
$R[\![\Gamma]\!]\otimes_{R,\phi}{\phi(R)'}  = \Lambda_{{\phi(R)}'}$ is a
multiple of the $p$-adic $L$-function $\CL^\Sigma_{f_\phi}$ by a unit in $\phi(R)'$.
\end{prop}
\noindent Here ${\phi(R)}'$ is the integral closure of $\phi(R)$ in its field of fractions (which
is a finite extension of $L$). In particular, as $\phi_{k_m}(R) = \O$, the image of $\CL^\Sigma_\ff$ in 
$R[\![\Gamma]\!]\otimes_{R,\phi_{k_m}} \O = \Lambda_\O$ is just $u_m\CL^\Sigma_{f_m}$ for 
some $u_m\in \O^\times$. Assuming that $\bar\rho_f$ is irreducible, for each $m$
we then have an equality of $\Lambda_\O$-ideals
\begin{equation}\label{padicLequation}
(\CL^\Sigma_f,p^m) = (\CL^\Sigma_{f_m},p^m) \subseteq \Lambda_\O.
\end{equation}

\section{Assembling the pieces}\label{assemble}

We can now put together the various objects and results from Section \ref{gather} to prove 
Theorems \ref{thmA} and \ref{thmB} as indicated in the introduction.  We will freely use the notation introduced in Section \ref{gather}.

\subsection{Proof of Theorem A}
Let $f$, $L$, $\O$ be as in the statement of Theorem \ref{thmA}. In particular, 
$f\in S_k(N)$ is a newform of some weight $k\geq 2$ that is conguent to 2 modulo $p-1$
and some level $N$. Furthermore, if $f=\sum_{n=1}^\infty a_n(f)q^n$ is the $q$-expansion
of $f$, then $a_p(f)\in \O^\times$. If $p\nmid N$, then by Theorem \ref{MCthm1} the Iwasawa--Greenberg Main Conjecture is true: for any finite set of primes $\Sigma$ containing $p$, $Ch_L^\Sigma(f) = (\CL_f^\Sigma)$ in $\Lambda_\O$. So we assume that $p\mid N$. By Lemma \ref{ordlemma} we then have $N=pM$ with $p\nmid M$ and $k=2$.  Let $T_f\subset V_f$ be a $G_\Q$-stable $\O$-lattice.
By Lemma \ref{lattice-lem} this lattice is unique up to $L^\times$-multiple since $\bar\rho_f$ is assumed irreducible.

Let $\Sigma\supset \{\ell\mid N\}$ be a finite set of primes.
After possibly replacing $L$ with a finite extension, for each integer $m>0$ there exists
\begin{itemize}
\item[(a)] a newform $f_m\in S_{k_m}(\Gamma_0(M))$ with $\Q(f_m)\subset L$, $k_m>2$, and $k_m\equiv 2 \pmod{p-1}$ and such that $a_p(f_m) \in \O^\times$;
\item[(b)] a $G_\Q$-stable $\O$-lattice $T_{f_m}\subset V_{f_m}$ and 
an isomorphism $T_f/p^m T_f \cong T_{f_m}/p^m T_{f_m}$ as $\O[G_\Q]$-modules that identifies
$T^+_f/p^m T_f^+$ with $T_{f_m}^+/p^m T_{f_m}^+$ as $\O[G_{\Qp}]$-modules;
\item[(c)] an equality of ideals $(\CL_f^\Sigma,p^m) = (\CL_{f_m}^\Sigma,p^m)\subseteq \Lambda_\O$.
\end{itemize}
The forms $f_m$ in (a) are just those defined in the discussion of Hida families in Section \ref{HidaThy}. Then (b) is just \eqref{latticeequality} and \eqref{latticeequality2}, and (c) is
\eqref{padicLequation}.  Furthermore, we also have
\begin{itemize}
\item[(d)] $\bar\rho_{f_m} \cong \bar\rho_f$ is irreducible and ramified at some $q\neq p$ such 
that $q\mid\mid M$;
\item[(e)] $X^\Sigma_{\Q_\infty,L}(f_m)$ is a torsion $\Lambda_\O$-module 
and $Ch^\Sigma_L(f_m) = (\CL_{f_m}^\Sigma) \subseteq \Lambda_\O$.
\item[(f)] $X^\Sigma_{\Q_\infty,L}(f_m)$ has no nonzero finite-order
$\Lambda_\O$-submodules, so $F^\Sigma_L(f_m) = Ch^\Sigma(f_m)$.
\end{itemize}
Note that (d) follows from (b) and the hypotheses on $N$ and $\bar\rho_f$ in Theorem \ref{thmA},
while (e) and (f) follow from the Iwasawa--Greenberg Main Conjecture for $f_m$ (which holds by (a), (d), and Theorem \ref{MCthm1} since $f_m$ is of level $M$ and $p\nmid M$) together with Proposition \ref{Selmerstructureprop2} and Lemma \ref{Ch-Fitt-coro}.

From (b) together with Lemma \ref{Selmer-modpm} we conclude that there is a $\Lambda_\O$-isomorphism
$$
\Sel^\Sigma_{\Q_\infty,L}(f)[p^m] \cong \Sel^\Sigma_{\Q_\infty,L}(f_m)[p^m]
$$
of $\Lambda_\O$-modules, and hence, upon taking Pontryagin duals, also a
$\Lambda_\O$-isomorphism
$$
X^\Sigma_{\Q_\infty,L}(f)/ p^m X^\Sigma_{\Q_\infty,L}(f) 
\cong X^\Sigma_{\Q_\infty,L}(f_m)/ p^m X^\Sigma_{\Q_\infty,L}(f_m).
$$
From basic properties of Fitting ideals we then conclude that there as an equality
of $\Lambda_\O$-ideals
$$
(F^\Sigma_L(f),p^m) = (F^\Sigma_L(f_m),p^m).
$$
Together with (c), (e), and (f) we then have
\begin{equation}\label{Fitt-L-eq}
(F^\Sigma_L(f),p^m) = (\CL_f^\Sigma, p^m) \subseteq \Lambda_\O.
\end{equation}

As $\CL_f$, and hence $\CL_f^\Sigma$, is non-zero by a well-known theorem
of Rohrlich \cite[Thm.~1]{Rohr}, if $m$ is large enough then $(\CL_f^\Sigma,p^m) \neq p^m\Lambda_\O$.
From this and \eqref{Fitt-L-eq} it then follows that if $m$ is large enough, then
$(F_L^\Sigma(f),p^m) \neq p^m\Lambda_\O$ and hence $F_L^\Sigma(f)\neq 0$.
As $F_L^\Sigma(f)\neq 0$, $X_{\Q_\infty,L}^\Sigma(f)$ must be a torsion $\Lambda_\O$-module.
It then follows from Proposition \ref{Selmerstructureprop2}(ii) and Lemma \ref{Ch-Fitt-coro}(ii)
that $Ch_L^\Sigma(f) = F_L^\Sigma(f)$. Combining this with \eqref{Fitt-L-eq} we then 
conclude that for all integers $m$
\begin{equation}\label{Ch-L-eq1}
(Ch_L^\Sigma(f), p^m) = (\CL_f^\Sigma, p^m) \subseteq \Lambda_\O.
\end{equation}

The characteristic ideal $Ch_L^\Sigma(f)$ is a principal ideal. Let $\CC_f^\Sigma$ be
a generator. 
From \eqref{Ch-L-eq1} it follows that for each integer $m$ there is an $u_m\in \Lambda_\O$ such that
\begin{equation}\label{modpmeq1}
\CC_f^\Sigma - u_m\CL_f^\Sigma \in p^m\Lambda_\O.
\end{equation}
Let $\varpi$ be a uniformizer of $\O$ and let $e$ be such that $(p) = (\varpi^e)$. As $\CL_f^\Sigma\neq 0$, there exists an integer $m_0\geq 0$ such that $\CL_f^\Sigma(f)\in \varpi^{m_0}\Lambda_\O$ but $\CL_f^\Sigma(f)\not\in \varpi^{m_0+1}\Lambda_\O$. It then follows from \eqref{modpmeq1}
that 
\begin{equation*}\label{modpmeq2}
u_{m'} - u_m \in \varpi^{me-m_0}\Lambda_\O, \ \ m'\geq m.
\end{equation*}
Therefore the sequence $\{u_m\}$ converges in $\Lambda_\O$ to an element $u\in \Lambda_\O$
such that for all $m$
$u-u_m \in \varpi^{me-m_0}\Lambda_\O.$
From this and \eqref{modpmeq1} it follows that 
$$
\CC^\Sigma_f -u \CL_f^\Sigma \in \varpi^{me-m_0} \ \ \forall m\geq 0,
$$
whence $\CC^\Sigma_f = u \CL^\Sigma_f$. That is $\CC^\Sigma_f \in (\CL_f^\Sigma)$. 

Since $X^\Sigma_{\Q_\infty,L}(f)$ is a torsion $\Lambda_\O$-module, $Ch_L^\Sigma(f)$ is non-zero, and so $\CC_f^\Sigma\neq 0$. We may then reverse the roles of $\CC^\Sigma_f$ and $\CL_f^\Sigma$ in the above argument to show that $\CL^\Sigma_f\in (\CC^\Sigma_f)$. From the two inclusions we then conclude 
$$
(\CL^\Sigma_f) = (\CC^\Sigma_f) = Ch_L^\Sigma \subseteq \Lambda_\O.
$$
This proves the desired equality, at least for the chosen $L$ and for $\Sigma$ containing all primes $\ell\mid N$. But, as observed in Section \ref{MCsection}, this implies the desired equality for all sets 
$\Sigma$ and all possible $L$. That is, the Iwasawa--Greenberg Main Conjecture holds for $f$: 
Theorem \ref{MCthm1} holds without hypothesis (iv).

\subsection{Proof of Theorem B} 
Let $f$, $L$, $\O$ be as in the statement of Theorem \ref{thmB}. 
As these also satisfy the hypotheses of Theorem \ref{thmA}, $X_{\Q_\infty,L}(f)$ is a torsion $\Lambda_\O$-module and 
its $\Lambda_\O$-characteristic ideal $Ch_L(f)$ is generated by the $p$-adic $L$-function
$\CL_f$. Furthermore, by Proposition \ref{Selmerstructureprop2}, neither $X_{\Q_\infty,L}(f)$ nor $\CX$ have
a nonzero finite-order $\Lambda_\O$-submodule.  To deduce the conclusions of Theorem \ref{thmB} from 
this, we make a close study of $\Sel_{\Q^\infty}(f)[\gamma-1]$
and $\CS[\gamma-1]$, following the methods of Greenberg \cite{Gr-CIME}.  

By Proposition \ref{Selmerstructureprop1} there is an exact sequence
$$
0 \rightarrow \CS \rightarrow \Sel_{\Q_\infty,L}(f) \rightarrow H^1(\F_p,(\CM^-)^{I_p}) \rightarrow 0.
$$
As $G_\Qp$ acts on $\CM^- \cong \Lambda^*$ through the character $\alpha_f \Psi^{-1}$,
$(\CM^-)^{I_p} \cong \Lambda^*[\gamma-1]=\Hom_\Zp(\O,\Qp/\Zp) \cong L/\O$ with $G_\Qp$ acting through the unramified character $\alpha_f$. 
Let
$$
\alpha_p = \alpha_f(\frob_p). 
$$
Then 
$H^1(\F_p,(\CM^-)^{I_p} )=0$ unless $\alpha_p=1$ (i.e., unless $f$ has split multiplicative reduction at $p$), in which case
it is isomorphic to $L/\O$. Letting $Ch_L(f)'$ be the $\Lambda_\O$-characteristic ideal of $\CX$, it follows that
$$
Ch_L(f) = Ch_L(f)' \cdot \begin{cases} (\gamma-1) & \text{$f$ has split multiplicative reduction at $p$}  \\
1 & \text{otherwise.}
\end{cases}
$$
This reflects the `extra zero' phenomenon in the split multiplicative case observed at the end of Section \ref{padicLs}.
In fact, we then have 
\begin{equation*}\label{charSeq1}
Ch_L(f)' = \begin{cases} 
(\CL_f') & \text{$f$ has split multiplicative reduction at $p$} \\
(\CL_f) & \text{otherwise.}
\end{cases}
\end{equation*}
As $\CX$ has no non-zero finite-order $\Lambda_\O$-submodules, a standard result\footnote{See footnote 5.} in Iwasawa theory  
gives $\#\CX/(\gamma-1)\CX = \#\Lambda_\O/(\gamma-1,Ch_L(f)')$.  As $\#\CS[\gamma-1] = \#\CX/(\gamma-1)\CX$, we then find
\begin{equation}\label{CSordereq1}
\#\CS[\gamma-1] = \begin{cases} \#\O/(\CL_f'(\phi_0)) & \text{$f$ has split multiplicative reduction at $p$} \\
\#\O/(\CL_f(\phi_0)) & \text{otherwise,}
\end{cases}
\end{equation}
where $\phi_0:\Lambda_\O\rightarrow\O$ is the continuous $\O$-algebra homomorphism sending $\gamma$ to $1$.

Let $\Sigma = \{\ell\mid Np\}$. Let
$$
W = \CM[\gamma-1] \cong T_f\otimes_\Zp\Qp/\Zp \ \ \text{and} \ \
W^\pm = \CM[\gamma-1]^\pm \cong T_f^\pm\otimes_\Zp\Qp/\Zp.
$$
Let
$$
\CP_\Sigma = H^1(\Q_p,\CM^-) \times \prod_{\ell\in\Sigma, \ell\neq p}H^1(\Ql,\CM)
$$
and
$$
P_{\Sigma} = H^1(\Qp,W)/L_p(W) \times \prod_{\ell\in\Sigma,\ell\neq p}H^1(\Ql,W),
$$
where $L_p(W) = \im\{H^1(\Qp,W^+)\rightarrow H^1(\Qp,W)\}$. Let $P_{\Sigma}^\div$ be defined
just as $P_{\Sigma}$ but with $L_p(W)$ 
replaced by its maximal divisible subgroup $L_p(W)^{\div}$. The usual (torsion) Bloch-Kato Selmer group for $T_f$ is just
\begin{equation*}\label{Sel-f-eq}
\Sel_L(f) = \ker\{H^1(G_\Sigma, W) \stackrel{res}{\rightarrow} P_{\Sigma}^\div\}.
\end{equation*}
As the the restriction map $H^1(G_\Sigma,\CM) \rightarrow \CP_\Sigma$
is surjective by Proposition \ref{Selmerstructureprop1}, we conclude that there is a short exact sequence
$$
0\rightarrow \Sel_L(f) \rightarrow \CS[\gamma-1] \rightarrow \im\{H^1(G_\Sigma,W)\stackrel{res}{\rightarrow} P_{\Sigma}^\div\}\cap
\ker\{P_{\Sigma}^\div\rightarrow \CP_\Sigma[\gamma-1]\} \rightarrow 0.
$$
Let $K= \ker\{P_\Sigma^\div\rightarrow \CP_\Sigma[\gamma-1]\}$. We claim that
\begin{equation}\label{CSordereq2} 
\#\CS[\gamma-1] = \#\Sel_L(f)\cdot \#K. 
\end{equation}
If $\Sel_L(f)$ is infinite, then there is nothing to prove since $\Sel_L(f)\subset \CS[\gamma-1]$. 
Suppose then that $\Sel_L(f)$ is finite. We will show that  the restriction map
$H^1(G_\Sigma,W)\stackrel{res}{\rightarrow} P_{\Sigma}^\div$ is surjective, from which the claim follows.

By global duality, the cokernel
of the restriction map $H^1(G_\Sigma,W){\rightarrow} P_{\Sigma}^\div$ is dual to a subquotient of 
$$
\Sel^\Sigma(T_f)^\sat = \ker\{H^1(G_\Sigma,T_f)\rightarrow H^1(\Qp,T_f)/L_p(T_f)^\sat\},
$$ 
where $L_p(T_f) = \im\{H^1(\Qp,T_f^+)\rightarrow H^1(\Qp,T_f)\}$ and 
$$
L_p(T_f)^\sat = \{x \in H^1(\Qp,T_f) \ : \ p^n x\in L_p(T_f) \ \text{some $n\geq 0$}\}.
$$
Here we have used that $T_f\cong \Hom_\Zp(W,\Qp/\Zp(1))$ as an $\O[G_\Q]$-module and that such an isomorphism 
identifies $L_p(T_f)^\sat$ and $L_p(W)^\div$
as mutual annihilators under local Tate duality. Then
$\Sel^\Sigma(T_f)^\sat$ is a torsion-free $\O$-module (as $\bar\rho_f$ is irreducible) 
and its $\O$-rank equals the $\O$-corank of $\Sel_L(f)$.
In fact, $\Sel^\Sigma(T_f)^\sat = H^1(G_\Sigma, T_f) \cap H^1_f(\Q,V_f)$, where 
$$
H^1_f(\Q,V_f) = \ker\{H^1(G_\Sigma, V_f)\stackrel{res}{\rightarrow} H^1(\Qp,V_f)/L_p(V_f) \times \prod_{\ell\in\Sigma,\ell\neq p}
H^1(\Ql,V_f)\}
$$
and $L_p(V_f) = \im\{H^1(\Qp,V_f^+)\rightarrow H^1(\Qp,V_f)\}$. (So $H_f^1(\Q,V_f)$ is just the usual characteristic zero 
Bloch-Kato Selmer group of $V_f$.)  In particular, the $\O$-rank of $\Sel^\Sigma(T_f)^\sat$ is the $L$-dimension 
of $H^1_f(\Q,V_f)$. The image of $H^1_f(\Q,V_f)$ in $H^1(G_\Sigma,T_f\otimes_\Zp\Qp/\Zp)\cong H^1(G_\Sigma,W)$
is the maximal divisible submodule of $\Sel_L(f)$. However, the latter is assumed to be of finite-order, so its maximal divisible 
subgroup is trivial. This proves that $\Sel^\Sigma(T_f)^\sat = 0$ and hence that the restriction map
$H^1(G_\Sigma,W)\stackrel{res}{\rightarrow} P_{\Sigma,x}^\div$ is a surjection. The equality \eqref{CSordereq2} follows.

Put
$$
L^\alg(f,1) = \frac{L(f,1)}{-2\pi i \Omega_f^+}.
$$
Combining \eqref{CSordereq1} with \eqref{CSordereq2}, the specialization formula for $\CL_f$, and the Greenberg-Stevens
formula \eqref{GrSteq} yields
\begin{equation}\label{Selmerordereq1}
\#\Sel_L(f)\cdot \#K = \begin{cases}
\#\O/(\frac{1}{\log_pu}\cdot\grL(V_f)\cdot L^\alg(f,1)) & \alpha_p=1 \\
\#\O((1-\alpha_p)^2\cdot L^\alg(f,1)) & \text{otherwise}.
\end{cases}
\end{equation}
Therefore,  to complete the proof Theorem \ref{thmB} it remains to express $\#K$ in terms of Tamagawa factors and the $L$-invariant $\grL(V_f)$.

From the definition of $K$, 
\begin{equation}\label{Kprodeq}
K= \prod_{\ell\in\Sigma} K_\ell
\end{equation}
with 
$$
K_\ell = \begin{cases}
\ker\{H^1(\Ql,W){\rightarrow} H^1(\Ql,\CM)\} & \ell\neq p \\
\ker\{H^1(\Qp,W)/L_p(W)^\div\rightarrow H^1(\Qp,\CM^-)\} & \ell=p.
\end{cases}
$$
If $\ell\neq p$,  then $\CM^{I_\ell}$ is $\gamma-1$-divisible and so $H^1(I_\ell,W)\hookrightarrow H^1(I_\ell,\CM)$ and 
$$
K_\ell = \ker\{H^1(\F_\ell,W^{I_\ell})\rightarrow H^1(\F_\ell,\CM^{I_\ell})=0\} = H^1(\F_\ell,W^{I_\ell}).
$$
Therefore
\begin{equation}\label{Kleq}
\#K_\ell = \#H^1(\F_\ell,W^{I_\ell})= c_\ell(T_f),
\end{equation}
where $c_\ell(T_f)= \#H^1(\F_\ell,W^{I_\ell})$ is just the Tamagawa number at $\ell\neq p$ defined by Bloch and Kato for the 
$p$-adic representation 
$T_f$. Note that $c_\ell(T_f)=1$ if $\ell\nmid N$ (i.e., if $T_f$ is unramified at $\ell$).
Hence to complete the proof of Theorem \ref{thmB} it remains to express $\#K_p$ in terms
of $\alpha_p$ if $f$ does not have split multiplicative reduction at $p$ (equivalently $\alpha_p\neq 1$) and in terms of 
$\grL(V_f)$ and the Tamagawa number at $p$ of $T_f$ otherwise.

Let 
$$
c_p' = \#\ker\{H^1(\Qp,W)/L_p(W)^\div\twoheadrightarrow H^1(\Qp,W)/ L_p(W)\}
$$
and
$$
c_p'' = \#\ker\{H^1(\Qp,W)/L_p(W)\rightarrow H^1(\Qp,\CM)/ L_p(\CM)\}.
$$
Then 
$$
\# K_p = c_p'c_p''.
$$

By Tate local duality, 
$L_p(W)$ is dual to $H^1(\Qp,T_f)/L_p(T_f)$ and $L_p(W)^\div$ is dual to $H^1(\Qp,T_f)/L_p(T_f)^\sat$.
Therefore
\begin{equation*}
c_p'  = \#(L_p(W)/L_p(W)^\div)  = \#(L_p(T_f)^\sat/L_p(T_f)).
\end{equation*}
Since $H^1(\Qp,T_f)/L_p(T_f)\hookrightarrow H^1(\Qp,T_f^-)$ and 
$H^1(\Qp,T_f)/L_p(T_f)^\sat\hookrightarrow H^1(\Qp, V_f^-)$, we find that the (injective) image of 
$L_p(T_f)^\sat/L_p(T_f)$ in $H^1(\Qp,T_f^-)$ is just
$$
\im\{H^1(\Qp,T_f)/L_p(T_f)\hookrightarrow H^1(\Qp,T_f^-)\}\cap \ker\{H^1(\Qp,T_f^-)\rightarrow H^1(\Qp,V_f^-)\}.
$$
But $H^1(I_p,T_f^-) \hookrightarrow H^1(I_p,V_f^-)$, so 
$\ker\{H^1(\Qp,T_f^-)\rightarrow H^1(\Qp,V_f^-)\} = H^1(\F_p, T_f^-)$. On the other hand, the boundary map injects the cokernel of 
$H^1(\Qp,T_f)/L_p(T_f)\hookrightarrow H^1(\Qp,T_f^-)$ into $H^2(\Qp, T_f^+)$ but
sends the the subgroup $H^1(\F_p,T_f^-)$  to zero (since $\Gal(\bar\F_p/\F_p)$ has cohomological dimension one). 
Hence the image of $H^1(\Qp,T_f)/L_p(T_f)\hookrightarrow H^1(\Qp,T_f^-)$ contains $H^1(\F_p,T_f^-)$.
It then follows that
$$
L_p(T_f)^\sat/L_p(T_f) \isoarrow H^1(\F_p,T_f^-) \cong \begin{cases} 0 & \alpha_p =1 \\ 
\O/(\alpha_p-1) & \text{otherwise}.
\end{cases}
$$
In particular,
\begin{equation*}\label{cp1}
c_p' = \begin{cases} 1 & \alpha_p=1 \\  \#(\O/(\alpha_p-1)) & \text{otherwise}.
\end{cases}
\end{equation*}

It remains to deduce the desired expression for $c_p''$.  By definition $c_p''$ equals
$$
\#\left(\im\{H^1(\Qp,W)/L_p(W)\hookrightarrow H^1(\Qp,W^-)\}\cap \ker\{H^1(\Qp,W^-)\rightarrow H^1(\Qp,\CM^-)\}\right).
$$
Since $H^2(\Qp,W^+)$ is dual to $H^0(\Qp,T_f^-)$ and the latter is $0$ 
if $\alpha_p\neq 1$, we have $H^1(\Qp,W)/L_p(W)\isoarrow H^1(\Qp,W^-)$ if $\alpha_p\neq 1$. 
It follows that in this case
$$
c_p'' = \#\ker\{H^1(\Qp,W^-)\rightarrow H^1(\Qp,\CM^-)\}= \#(\CM^-)^{G_\Qp}/(\gamma-1)\cdot(\CM^-)^{G_\Qp}.
$$
As $I_p$ acts on $\CM^-$ through the character $\Psi^{-1}$ and $\frob_p$ acts on $(\CM^-)^{I_p} = \CM^-[\gamma-1] \cong L/\O$ as 
multiplication by $\alpha_p$, we find
$$
 c_p'' = \#L/\O[\alpha_p-1] = \# \O/(\alpha_p-1), \ \ \alpha_p\neq 1.
$$

Suppose then that $\alpha_p=1$.  It follows from local duality that $c_p$ equals
the index of the $\O$-submodule of $H^1(\Qp,T_f^+)$ generated by $\ker\{H^1(\Qp,T_f^+) \hookrightarrow H^1(\Qp,T_f)\}$
and the annihilator of $\ker\{H^1(\Qp,W^-)\rightarrow H^1(\Qp,\CM^-)\}$. The first is just the image
of $\O\cong H^0(\Qp,T_f^-) \rightarrow H^1(\Qp,T_f^+)$ determined by the $G_\Qp$-extension $T_f$. Let $c_{V_f}$ be an 
$\O$-generator; this is a non-zero element in $\ell_{V_f}$ in the notation of Section \ref{Linvariant}. 
On the other hand, as $H^1(\Qp,W^-) \cong \Hom_{cts}(G^{ab,p}_\Qp,L/\O)$, the kernel
$\ker\{H^1(\Qp,W^-)\rightarrow H^1(\Qp,\CM^-)\}$ is readily seen to be $\Hom_{cts}(\Gamma,L/\O)$ - those homomorphisms that factor through $\Gamma$. Then, under the identification 
$$
H^1(\Qp,T_f^+) = H^1(\Qp,\O(1)) = (\varprojlim_n \Qp^\times/(\Qp^\times)^{p^n})\otimes_\Zp\O,
$$
the annihilator of $\Hom_{cts}(\Gamma,L/\O)$ is identified with the $\O$-module $p\otimes\O$ generated by the image of $p^\Z$. 
The index of $\O\cdot c_{V_f} + p\otimes\O$ is just the index of the projection of $c_{V_f}$ to $(\varprojlim_n \Z_p^\times/(\Z_p^\times)^{p^n})\otimes_\Zp \O$. From the definition of $\psi_\cyc$ in Section \ref{Linvariant}, this
index is just $\#\O/(\frac{1}{\log_p u}\cdot\psi_\cyc(c_{V_f}))$. So by the definition of $\psi_\ur$ (which is non-zero on $c_{V_f}$ as
$0\neq c_{V_f}\in \ell_{V_f}$) and the definition of $\grL(V_f)$,
$$
c_p'' = \#\O/(\frac{1}{\log_p u}\cdot\psi_\cyc(c_{V_f})) = \#\O/(\frac{1}{\log_p u}\cdot \grL(V_f)\cdot\psi_\ur(c_{V_f})), \ \ \alpha_p=1.
$$

Combining the formulas for $c_p''$ in the two cases with those for $c_p'$ we find
\begin{equation}\label{Kpeq}
\#K_p = \begin{cases} 
\#\O/(\frac{1}{\log_pu}\cdot \grL(V_f) \cdot \psi_\ur(c_{V_f})) & \alpha_p=1 \\
\#\O/(\alpha_p-1)^2 & \alpha_p\neq 1.
\end{cases}
\end{equation}

Suppose $\grL(V_f)\neq 0$ if $\alpha_p=1$. Then combining \eqref{Selmerordereq1} with \eqref{Kprodeq}, \eqref{Kleq},
and \eqref{Kpeq} yields
$$
\#\O/(L^\alg(f,1)) = \#\Sel_L(f)\cdot\prod_{\ell\neq p} c_\ell(T_f) \cdot\begin{cases}
\#\O/(\psi_\ur(c_{V_f})) & \alpha_p = 1 \\ 1 & \alpha_p\neq 1.
\end{cases}
$$
That the final term is just the Bloch-Kato Tamagawa number at $p$ of the representation $T_f$, which we denote $c_p(T_f)$,
can be shown as in \cite{DumTam}; in {\it loc.~cit.}~$c_p(T_f)$
is denoted $\mathrm{Tam}^0_M(T_f)$. The only significant change is the need to include the $\O$-action, but this is a straightforward modification. 
In the $p\nmid N$ case - that is, the case where $V_f$ is a crystalline representation of $G_\Qp$ - that $c_p(T_f)=1$ follows
by the arguments used to prove \cite[Thm.~5.1]{DumTam}. The $p\mid\mid N$ case - in which case $V_f$ is a semistable representation 
of $G_\Qp$ - follows as in \cite[\S7]{DumTam} from the arguments used to prove \cite[Thm.~6.1]{DumTam}.
We therefore have the formula asserted in Theorem \ref{thmB}:
\begin{equation}\label{Selmerordereq2}
\#\O/(L^\alg(f,1)) = \#\Sel_L(f)\cdot\prod_{\ell} c_\ell(T_f).
\end{equation}
This completes the proof of Theorem \ref{thmB}.

\subsection{Proof of Theorem \ref{thmC}}
Theorem \ref{thmC} is just a special case of Theorem \ref{thmB}. To see this, let $E$ be as in Theorem \ref{thmC} and let $f\in S_2(\Gamma_0(N))$ be the newform associated with $E$, so $N$ is the conductor of $E$ and $L(E,s) = L(f,s)$. For Theorem \ref{thmC} to follow from Theorem \ref{thmB}, if suffices to have that under the hypotheses of Theorem \ref{thmC}, hypotheses (i), (ii), and (iii) of Theorem \ref{thmB} hold for $f$ and 
$\Omega_E$ is a $\Z_{(p)}^\times$-multiple of $-2\pi i \Omega_f^+$. 

That hypotheses (i) and (ii) of Theorem \ref{thmC} imply hypotheses (i) and (ii) of Theorem \ref{thmB} is immediate.  Furthermore, as noted in the example at the end of 
\ref{Linvariant}, if $E$ has split multiplicative reduction at $p$ then the $\grL$-invariant $\grL(V_f)$ of $f$ is non-zero, hence hypothesis (iii)
of Theorem \ref{thmC} also holds. 

To compare periods, we first recall that if $\omega_E$ is a N\'eron differential of $E$ then 
$$
\Omega_E = \int_{c^+} \omega_E \in \C^\times,
$$
where $c^+$ is a generator of the submodule $H_1(E(\C),\Z)^+\subset H_1(E(\C),\Z)$
that is fixed by the action of $\Gal(\C/\R)$; this is well-defined up to multiplication by $\pm 1$. Now let
$$
\phi: X_1(N)\rightarrow E^\opt
$$
be an optimal parameterization for the $\Q$-isogeny class of $E$ as in \cite[Prop.~(1.4)]{Stevens-opt}. Then, as demonstrated in 
the proof of
\cite[Prop.~(3.1)]{Vat-Gr}, $\Omega_{E^\opt}$ equals $-2\pi i \Omega_f^+$ up to a $\Z_{(p)}^\times$-multiple\footnote{The key points
are \cite[Prop.~(3.3)]{Vat-Gr}, which shows that if $\ord_p(N)\leq 1$ then $\phi^*\omega_{E^\opt} = c\cdot 2\pi i f(z) dz$ for some integer $c\in \Z$ 
such that $p\nmid c$, and the fact - by the definition of an optimal parameterization - that $\phi$ induces a surjection $H_1(X_1(N)(\C),\Z)\twoheadrightarrow H_1(E^\opt(\C),\Z)$.}. 
Let 
$$
\beta:E^\opt\rightarrow E
$$
be a $\Q$-isogeny. Since $E[p]$ is an irreducible $G_\Q$-representation, $\beta$ can be chosen so that its degree is
prime to $p$. Then $\beta^*\omega_E$ is a $\Z_{(p)}^\times$-multiple of $\omega_{E^\opt}$, and so
$\Omega_E$ is a $\Z_{(p)}^\times$-multiple of $\Omega_{E^\opt}$ and hence also of $-2\pi i \Omega_f^+$.

\end{document}